\documentclass[12pt,leqno,draft]{amsart}  
\usepackage{backref}
\usepackage[inline,nomargin,index]{fixme}

\usepackage{amsmath,amstext,amsthm,amssymb,amsfonts}

\usepackage[colorlinks=true,pagebackref,hypertexnames=false, colorlinks, citecolor=red, linkcolor=red]{hyperref} 
\usepackage[backrefs]{amsrefs}

\usepackage{bbm} 

\numberwithin{equation}{section}

 \DeclareFontFamily{U}{wncy}{}
\DeclareFontShape{U}{wncy}{m}{n}{<->wncyr10}{}
\DeclareSymbolFont{mcy}{U}{wncy}{m}{n}
\DeclareMathSymbol{\Sh}{\mathord}{mcy}{"58}

\newcommand{\norm}[1]{\ensuremath{\|#1\|}}
\newcommand{\abs}[1]{\ensuremath{\vert#1\vert}}

\DeclareSymbolFont{bbold}{U}{bbold}{m}{n}
\DeclareSymbolFontAlphabet{\mathbbold}{bbold}

\allowdisplaybreaks
\swapnumbers
\theoremstyle{plain}

\newtheorem{lemma}[equation]{Lemma} 
\newtheorem{proposition}[equation]{Proposition} 
\newtheorem{theorem}[equation]{Theorem} 
 
\newtheorem{definition}[equation]{Definition}

\def\norm#1.#2.{\left\Vert#1\right\Vert_{#2}}
\def\Norm#1.#2.{\bigl\lVert#1\bigr\rVert_{#2}}
\def\NOrm#1.#2.{\Bigl\lVert#1\Bigr\rVert_{#2}}
\def\NORm#1.#2.{\biggl\lVert#1\biggr\rVert_{#2}}
\def\NORM#1.#2.{\Bigl\lVert#1\Bigr\rVert_{#2}}

\def\ip#1,#2,{\langle #1,#2\rangle}
\def\Ip#1,#2,{\bigl\langle#1,#2\bigr\rangle}
\def\IP#1,#2,{\Bigl\langle#1,#2\Bigr\rangle}

\def\mid{\,:\,}

\def\abs#1{\lvert#1\rvert}
\def\Abs#1{\bigl\lvert#1\bigr\rvert}

\def\R{\mathbb R}

\def\XXint#1#2#3{{\setbox0=\hbox{$#1{#2#3}{\int}$}
     \vcenter{\hbox{$#2#3$}}\kern-.5\wd0}}

\def\eqdef{\stackrel{\mathrm{def}}{{}={}}}

\def\schatten#1{\mathbb S^{#1}} 
\def\besov#1{ \mathbb B^p}

\begin{document}

\title[Paraproducts \& Schatten Norms]{Schatten Class Estimates for Paraproducts in  Multi-parameter setting} 


\author[M. T. Lacey]{Michael T. Lacey$^*$}

\address{Michael T. Lacey\\
School of Mathematics\\
Georgia Institute of Technology\\
Atlanta,  GA 30332 USA}


\email{lacey@math.gatech.edu}

\author[J. Li]{Ji Li$^\dagger$}

\address{Ji Li \\
School of Mathematical and Physical Sciences\\
Macquarie University\\
NSW 2109, Australia}


\email{ji.li@mq.edu.au}

\author[B. D. Wick]{Brett D. Wick$^\ddagger$}

\address{Brett D. Wick \\
Department of Mathematics\\
         Washington University - St. Louis\\
         St. Louis, MO 63130-4899 USA}


\email{bwick@wustl.edu}

\begin{abstract}
Let $\operatorname \Pi _{b}$ be a bounded $n$ parameter paraproduct with symbol $b$. We demonstrate that this operator is in the Schatten class $\schatten p$, $0<p<\infty$, if the symbol is in the $n$ parameter Besov space $\besov p$. Our result covers both the dyadic and continuous version of the paraproducts in the multiparameter setting. 
 \end{abstract}



\maketitle

 \section{Introduction, Notation and Statement of Main Results}

Paraproducts are an extremely useful tool in questions arising in harmonic analysis. They provide a nice class of singular integral operators, and when restricting to the dyadic case provide much insight into the mapping properties of Calder\'on--Zygmund operators.  Paraproducts have natural connections with other important operators in analysis.  In particular, it is possible to view paraproducts as either the commutator 
between a function and the Hilbert transform or equivalently a Hankel operator with a certain symbol.  It turns out the properties of the symbol heavily influence the operator theoretic characteristics of the paraproduct.  In this paper, we are interested in the property of the paraproduct (both continuous version and dyadic version) being in  certain Schatten class, with applications to commutators in multi-parameter settings which link to the big Hankel operators.


We note that along the line of  Calder\'on \cite{Cal}, Coifman--Rochberg--Weiss \cite{CRW}, Uchiyama \cite{U}, Janson--Wolff \cite{JW}, Rochberg--Semmes \cite{RochSem}, the theory of commutators (boundedness, compactness and Schatten class) plays an important role, which connects to the weak factorisation of Hardy space (\cite{CRW}) and Hankel operators (\cite{Pel,Bo,MR2097606}) in the complex analysis, compensated compactness in the PDEs \cite{CLMS}, as well as the quantised derivative in non-commutative analysis and geometry \cite{Connes,LMSZ}. 


To state the main results of this paper, let us recall the relevant paraproducts.  In so doing, we prefer a discrete formulation of these operators.  As is well known, there are two distinct ways to formulate them, with the Haar basis playing a distinguished role.  So, for the sake of definiteness, we first give the Haar paraproducts and the statements of the main results in this context.

\subsection*{One Parameter Paraproducts}
 We take the dyadic intervals to be  
 $$ 
 \mathcal D \eqdef \left\{[j2 ^{k}, (j+1)2 ^{k})\mid j,k\in \mathbb Z  \right\}.
 $$ 
 Each dyadic interval $I  $ is a union of its left and right halves, denoted $I _{-} $ and $I_+ $ respectively.  The Haar function $h_I $ adapted to $I $ is 
 \begin{equation}  
 \label{e.Haar}
 h^0_I {}\eqdef{}\abs I ^{-1/2} \left( -\mathbf 1 _{I_-}+\mathbf 1 _{I+} \right). 
 \end{equation}
 The other function of importance is  
 \begin{equation}  
 \label{e.zq}
 h^1 _{I} = \abs I ^{-1/2} \mathbf 1 _{I}. 
 \end{equation}
 Thus, $h^{0}_I $ has integral zero, while $h^1_I $ is a multiple of an indicator function. 
 The function $ h ^{1} _I$ is, in wavelet nomenclature, the `father' wavelet.

 The one parameter Haar paraproducts are 
 \begin{equation}  
 \label{e.Haar-para-one}
 \operatorname B_{\text{\rm Haar}} (f _{1},f _{2})\eqdef{}\sum _{I\in\mathcal D}
 \abs{I} ^{-1/2}h^1_I\prod _{j=1}^{2} \ip f _{j},h_{I}^0, _{L^2(\mathbb R)} .
 \end{equation}
 A classic result in dyadic harmonic analysis is that 
\begin{equation}
\label{e.Haar-BMO}
\norm  \operatorname B _{\rm Haar}(f_1,\cdot).2. {}\simeq{} 
\norm f_1.{\rm BMO} _{\text{\rm dyadic } }(\mathbb R) . .
\end{equation}
In this last display, $ {\rm BMO} _{\text{\rm dyadic } }(\mathbb R)$ is the dyadic  BMO space. See for example \cite{math.CA/0502334}.

To define more general paraproducts, we will appeal to \emph{wavelets}.  We make this more precise now.  For an interval $I $, we say that $\varphi $ is \emph {adapted to $I $} if and only if  $\norm \varphi.2.=1$ and 
\begin{equation} 
\label{e.adapted} 
\Abs{ {d^{\alpha}\over dx^\alpha} \varphi(x)} {}\lesssim{}\abs{I}^{-\alpha-\frac12} \left(  1+\tfrac {\displaystyle \abs{x-c(I)}}{\displaystyle \abs{I}}\right)^{-N}, 
 \qquad \alpha=0,1,2. 
 \end{equation}
 Here, $c(I) $ denotes the center of $I $, and $N $ is a large fixed integer, whose exact value  need not concern us.    By \emph{$\{\varphi_I\mid I\in\mathcal D\}$ are adapted to $\mathcal D $} we mean that for all dyadic intervals $I $, 
 $\varphi_I $ are adapted to $I $. 
 We shall consistently work with functions which are normalized in $L^2(\mathbb R) $.  Some of these functions we will 
 also insist to be normalized to have integral zero.   Indeed, by $\varphi_I $ \emph {has a zero} we mean that 
 $\varphi_I $ has integral zero. 

 A collection $\{\varphi_I\mid I\in\mathcal D\}$ is \emph{\textbf{uniformly adapted}} to $\mathcal D$ if for each $I\in\mathcal D$, $\varphi_I$
 is adapted to $I$, and each $\varphi_I $ is obtained from a single fixed function $\varphi\in C^\infty(\mathbb R) $ as follows
 \begin{equation} \label{e.uniformlyadapted}
 \varphi_I(x)=\frac1{\sqrt {\abs I } }\varphi\left(\tfrac{\displaystyle x-c(I) } {\displaystyle \abs I}\right), 
 \end{equation}
 where $c(I) $ is the center of $I $ and $\varphi$ satisfies
 $$ \int_0^\infty \big| \hat\varphi(t\xi)\big|^2{dt\over t}=1,\qquad \forall \xi\not=0. $$
   Most typically, the notation $\varphi _{I} $ will be used for a function adapted to $I $.  
 
 A collection of functions $\{w_I \mid I\in\mathcal D \} $  is called a \emph{wavelet basis} if the collection is uniformly adapted to $\mathcal D $, and it is an orthonormal basis for $L^2(\mathbb R)$.   It is very easy to see that necessarily, $w$ has integral zero.   The examples of wavelet bases that will be important for us are $L^2(\mathbb R)$ normalized functions adapted to an interval.

 Paraproduct operators are constructed from rank one operators $ f\mapsto \ip f,\varphi, \phi $.  
 A paraproduct is, in its simplest manifestation,
 of the form 
 \begin{equation*}
\operatorname B(f _{1},f _{2})\eqdef{}\sum _{I\in\mathcal D} \abs{I} ^{-1/2}\varphi _{3,I}\prod _{j=1}^{2} \ip f_{j},\varphi_{j,I}, _{L^2(\mathbb R)}.
 \end{equation*}
 Here, the functions $\varphi _{j,I}$, for $j=1,2,3 $ are adapted to $I$.  Two of these three functions are 
 assumed to be of integral zero, and in particular we will always assume that $\varphi _{{1,I} }$ has 
 integral zero.  We are concerned with extensions of the classical  conditions for this operator to 
be bounded on $L^2(\mathbb R)$.   For these more general paraproducts there is the following well-known extension of \eqref{e.Haar-BMO}.  See \cite{math.CA/0502334,MR2002m:47038} for proofs.


\begin{theorem}[\cite{math.CA/0502334}]
\label{t.one}
If $\{\varphi_{1,I }\} $ and at least one of $\{\varphi_{j,I}\} $ $j=2,3 $ have zeros and are adapted to $\mathcal D $, then 
\begin{equation}
\norm  \operatorname B(f_1,\cdot).L^2(\mathbb R)\to L^2(\mathbb R).  {}\lesssim{} 
\norm f_1.{\rm BMO}(\mathbb R) . .
\end{equation}
If all three collections $\{\varphi_{j,I}\} $ are uniformly adapted to $\mathcal D $, then the reverse inequality holds. 
 \end{theorem} 
 
The BMO  norm is explicitly given by
\begin{equation}  
\label{e.bmo}
\norm f.{\rm BMO}(\mathbb R). {}\eqdef{}\sup_U \left[ \abs U ^{-1 }\sum _{I\subset U} \abs{ \ip f, \varphi_I ,}^2 \right] ^{1/2 }.
\end{equation}
Here, the supremum is formed over all intervals $U$ and $\{\varphi_I\} $ is uniformly adapted to $\mathcal D $.  We note that the definition is independent of the choice of $\{\varphi_I\} $ and the dyadic grid $\mathcal D $.

In this paper we are principally concerned with 
the Schatten norms of paraproducts $\operatorname B(f _{1},f _{2})$ and its multiparameter versions.  Recall that the Schatten norm of an operator is given by a $\ell^p$ sum of its singular values, see Section \ref{s.schatten} for the precise definition and more information about these norms.  Much like the case of boundedness, membership of the paraproduct in a Schatten class can be characterized in terms of the function $f_1 $.  

This result has a long history.  In the case of continuous paraproducts, Janson and Peetre showed in \cite{MR924766} that membership in a Schatten class is equivalent to the symbol belonging to the Besov space.  Their method of proof was very much Fourier analytic by viewing the continuous paraproduct as a certain multiplier on the Fourier side and then decomposing the operator in a appropriate manner. See also Pott and Smith, \cite{MR2097606} for the dyadic version. Chao and Peng  \cite{MR1397362} showed  that the one parameter paraproducts arising from ($d$-dyadic) martingale transforms 
are bounded if and only if the symbol belongs to the dyadic Besov space, whose definition is given below.  The proof is very computational, and takes advantage of the notion of ``nearly weakly orthonormal sequences'' introduced by Rochberg and Semmes, \cite{RochSem}.  In fact, in both \cites{MR1397362, MR924766} it is shown that more generally the commutators with singular integral operators (or martingale transforms) belong to a certain Schatten class if and only if the symbol belongs to the appropriate Besov space.  

Membership of the paraproduct in the Schatten class is related to smoothness on the symbol $f_1$, and this is governed by the symbol belonging to a certain Besov space.  We can define the dyadic Besov spaces $\besov p _{\text{\rm dyadic}}(\mathbb R)$ as the set of $f\in L^1_{loc}(\mathbb R)$ such that $ \norm f.\besov p _{\text{\rm dyadic}}(\mathbb R).<\infty$, where
\begin{equation}  \label{e.besov-1d-dyadic}
 \norm f.\besov p _{\text{\rm dyadic}}(\mathbb R). {}\eqdef{}\left[\sum _{I\in\mathcal D } \left[ \abs I ^{-1/2}\abs{\ip f,h_I^0 , } \right] ^p \right] ^{1/p },
 \qquad 0<p<\infty.
 \end{equation}
 Also, the Besov space $\besov p(\mathbb R)$ is the set of Schwartz distributions $f$ such that $ \norm f.\besov p(\mathbb R).<\infty$, where
 \begin{equation}  
 \label{e.besov-one-def}
 \norm f.\besov p(\mathbb R). {}\eqdef{}\left[ \sum _{I\in\mathcal D } \left[ \abs I ^{-1/2}\abs{\ip f,\varphi_I , } \right] ^p \right] ^{1/p },
 \qquad 0<p<\infty
 \end{equation}
 with $\{\varphi_I\} $ is uniformly adapted to $\mathcal D $. This definition is independent of the choice of  $\{\varphi_I\}$ and the dyadic grid $\mathcal D$.

Our main result is then the following theorem, giving an extension and refinement of the boundedness of paraproducts in one parameter given in Theorem \ref{t.one}.

\begin{theorem}[Main Result 1]
\label{t.schatten-one}
Assuming zeros only for the functions $\{\varphi _{1,I }\mid I\in\mathcal D \} $, we have the estimate 
\begin{equation} 
\label{e.schatten-one} 
\norm  \operatorname B(f_1,\cdot). \schatten p.  {}\lesssim{}
 \norm f_1.\besov p(\mathbb R)., \qquad 0<p<\infty.  
\end{equation}
If each of the collections of functions $\{\varphi_{j,I}\} $, $j=1,2,3$, is uniformly adapted to $\mathcal D $, then the reverse inequality holds.
\end{theorem}

We have the following immediate corollary in the case of Haar paraproducts.
\begin{theorem}
\label{c.schatten-one}  
In the Haar case we have 
\begin{equation} 
\label{e.schatten-dyadic-one}
\norm  \sum _{I\in\mathcal D}  \frac{ \ip f_1,h_I^0, }{\sqrt{\abs I}} h^{\epsilon}_I\otimes h^\delta_I .\schatten p . \simeq \norm f_1.\besov p _{\text{\rm dyadic }}(\mathbb R). , \qquad 0<p<\infty, \ \{\epsilon,\delta\}\not=\{1,1\}.   
\end{equation}
\end{theorem} 
We remark that a variation of Theorem \ref{c.schatten-one} for dyadic martingale transforms was previously studied by Chao and Peng \cite{MR1397362}.  Additionally, a related proof of Theorem \ref{c.schatten-one} was given by Pott and Smith \cite{MR2097606}.

We will now use these ideas to extend the results from one-parameter to the multi-parameter setting.

\subsection*{Multi-parameter paraproducts}
We use many of the ideas from the previous section to form the tensor product basis in $L^2(\mathbb R^n)=L^2(\mathbb R\times\cdots\times\mathbb R)$.  Let $ \mathcal R  \eqdef \mathcal D\times \cdots \times \mathcal D$ be the collection of dyadic rectangles in $ \mathbb R ^{n}$.  

\noindent$\bullet$ {Dyadic version of paraproduct:}

For a rectangle $R = R_1\times \cdots \times R_n\in\mathcal R$ 
and for a choice of $\varepsilon=(\varepsilon_1,\ldots,\varepsilon_n)\in\{0, 1\}^n$, 
\begin{equation}
\label{e.Haar-nd-dim}
 h_R^{\varepsilon }(x_1,\ldots,x_n) \eqdef{}\prod_{j=1}^n h _{R_j}^{\varepsilon_j}(x_j),
\end{equation}
where $h_{R_j}^{\varepsilon_j}$ is given by either \eqref{e.Haar} or \eqref{e.zq} depending on if $\varepsilon_j=0$ or $\varepsilon_j=1$.
For $ \varepsilon =\vec 0= (0,\dotsc,0)$, we  
denote $ h_R^{\vec0 }(x_1,\ldots,x_n)$ by  $h_R(x_1,\ldots,x_n)$ for simplicity.

Set $E_n\eqdef{}\{\varepsilon\in\{0,1\}^n\}\setminus\vec{1}$,  where $\vec{1}=(1,\ldots,1)$. Note that the cardinality of the set $E_n$ is $2^n-1$.  We then have that $\{h_R^\varepsilon:R\in\mathcal R,\varepsilon\in E_n\}$ is the product Haar basis for $L^2(\mathbb R^n)$. 

The $n$ parameter Haar paraproducts are 
\begin{equation}
\label{e.dparameterHaar}
\operatorname B _{\textup{Haar}} (f_1,f_2) \eqdef \sum _{R\in \mathcal R } \frac {\ip f_1, h_R,} {\sqrt {\abs{ R}}}  \ip f_2, h ^{\varepsilon }_R, _{L^2(\mathbb R^n)}\, h _{R} ^{\delta }.
\end{equation}
The essential restriction to place on the two choices of $ \varepsilon ,\delta \in \{0,1\} ^{n}$ is that 
$$
\emph {in no coordinate  $ j$ do we have $ \varepsilon _{j}=\delta _{j}=1$.} 
$$
 Observe that this condition permits a wide variety of paraproducts, most with most having no proper analog as compared to the one dimensional case. 
 
\noindent$\bullet$ {Continuous version of paraproduct:}

Let us say that a collection $\{\varphi_R: R\in\mathcal R\} $  \emph {is adapted to a rectangle $R=R _{1}\times \cdots\times R _{n}$} if and only if $\varphi_R (x_1,\ldots, x_n)=\prod_{j=1}^n\varphi_{j,R_j}(x_j)$, with $R=R_1\times\cdots\times R_n $ and each $\{\varphi_{j, R_j }\}$ is adapted to $\mathcal D$.  We say that $\{\varphi_R \} $ has \emph {zeros in the $j $th coordinate} if and only if $\{\varphi_{j,R_j} \} $ has zeros.  The collection $\{\varphi_R\}$ is \emph{uniformly adapted to $\mathcal R$} if and only if each $\{\varphi_{j, R_j}\}$, $j=1,\ldots,n$, is uniformly adapted to $\mathcal D$. 

Paraproducts  are then defined as follows: 
\begin{equation}\label{mul paraproduct}
\operatorname B(f _{1},f _{2})\eqdef{}\sum _{R\in\mathcal R} \frac{\varphi_{R}^{(3)}}{\sqrt{\abs{R} }}\prod _{j=1}^{2} \ip f _{j},\varphi_{R}^{(j)},_{L^2(\mathbb R^n)},
 \end{equation}
where the functions $\varphi_{R}^{(j)}$  are adapted to $R$ for $j=1,2,3 $. 
The construction of a smooth wavelet basis in $L^2(\mathbb R^n)$ is similar and standard. For the details we omit here.

The boundedness of the multi-parameter paraproducts was first studied by Journ\'e when considering the $T(1)$ Theorem in the product setting, \cite{MR88d:42028}.  These results were later studied further by Muscalu, Pipher, Tao and Thiele in the following papers \cites{camil, math.CA/0411607} which showed the richness of the paraproduct structures in the multiparameter setting.  One should also see the article by Lacey and Metcalfe, \cite{math.CA/0502334}.  The following general theorem on the boundedness in $L^2(\mathbb R^n)$ of the Haar paraproducts and more general paraproducts from a wavelet basis is then given by:
\begin{theorem}[\cites{MR88d:42028,camil, math.CA/0411607, math.CA/0502334}]
For the multi-parameter Haar paraproducts, we have the following estimate:
\label{t.2}
\begin{equation}
\label{t.dyadicnparameter} 
\norm \operatorname B _{\textup{Haar}} (f_1,\cdot).L^2(\mathbb R^n)\to L^2(\mathbb R^n). \lesssim \norm f_1. {\rm BMO} _{\textup{dyadic}}(\mathbb R\times\cdots\times\mathbb R). .
\end{equation}

\noindent More generally, assume that for both coordinates $j=1,2 $ there is a choice of $k\in\{2,3\} $ for which $\varphi_{R}^{(1)}$ and  $\varphi_{R}^{(k)} $ zeros in the $j$th coordinate. Then, for the paraproducts as in \eqref{mul paraproduct}, we have the inequality: 
\begin{equation}  
\label{e.B-2-}
\norm \operatorname B(f_1,\cdot).L^2(\mathbb R^n)\to L^2(\mathbb R^n). {}\lesssim{}\norm f_1.{\rm BMO}(\mathbb R\times\cdots\times\mathbb R). .
\end{equation}
Here ${\rm BMO} (\mathbb R\times\cdots\times\mathbb R)$ is the product {\rm BMO} space studied by  S.-Y. A. ~Chang and R.~Fefferman, \cite{cf2}. 
\end{theorem}  

There are two points to make about this last inequality.  The first is that the  BMO norm is given explicitly by
 \begin{equation}  \label{e.bmocf}
 \norm f.{\rm BMO} (\mathbb R\times\cdots\times\mathbb R). {}\eqdef{}\sup_U \left[ \abs U ^{-1 }\sum _{R\subset U} \abs{ \ip f, w_R , }^2 \right] ^{1/2 }.
 \end{equation}
 Here, the supremum is formed over \textit{open} sets $U$ and $\{w_R\} $ is a product wavelet basis.  Replacing the wavelet basis by the Haar basis, we have dyadic Chang--Fefferman BMO, \cite{cf2}.  The second is that we are not asserting the equivalence of norms. Indeed, for a `degenerate' $n$ parameter paraproducts, the equivalence of norms is not so clear.   There are essentially two distinct cases.  The first case, with the greatest similarity to the one parameter case, is where we have for example, $\{ \varphi_{R}^{(2)} \}$ has  zeros in all coordinates.  The second case with no proper analog in the one variable setting is,  for instance, $\{ \varphi_{R}^{(2)}\} $ has zeros in one set of coordinates while $\{\varphi_{R}^{(3)}\} $ has zeros in a complimentary set of coordinates. 

Similar to  the one-parameter case, with the Haar basis we can define the dyadic (product) Besov spaces in $\mathbb R^n$ as
\begin{equation}  \label{e.besov-nd-dyadic}
 \norm f.\besov p _{\text{\rm dyadic}}(\mathbb R\times\cdots\times\mathbb R). {}\eqdef{}\left[ \sum _{R\in\mathcal R} \left[ \abs R ^{-1/2}\abs{\ip f,h_R, } \right] ^p \right] ^{1/p },
 \qquad 0<p<\infty.
 \end{equation}
For the tensor product wavelet basis, we define the (product) Besov spaces in $\mathbb R^n$ as
 \begin{equation}  
\label{e.besov-nd-def}
 \norm f.\besov p(\mathbb R\times\cdots\times\mathbb R). {}\eqdef{}\left[ \sum _{R\in\mathcal R} \left[ \abs R^{-1/2}\abs{\ip f,\varphi_R , } \right] ^p \right] ^{1/p },
 \qquad 0<p<\infty.
 \end{equation}
 Again, we will see that this definition does not depend upon the choice of wavelet basis.

Our principal estimate, giving an extension of Theorem \ref{t.2} to the multi-parameter setting is given next. For simplicity, this theorem is stated in the case of two parameters.  The correct statement of the general multi-parameter version can be obtained from Theorem \ref{t.Haardparproducts} below.
\begin{theorem}[Main Result 2]
\label{t.two-schatten}
Assume  that $\{\phi_{1,R }\} $ has cancellation for both coordinates $j=1,2$, while $\{\phi_{2,R }\} $ and $\{\phi_{3,R }\} $ have the property that: if $\{\phi_{2,R }\} $ has a zero in coordinate $j$ then $\{\phi_{3,R }\} $ does not and vice versa, for $j=1,2$. Then, 
\begin{equation}  
\label{e.two-schatten}
\norm \operatorname B(f_1,\cdot) . \schatten p. {}\lesssim{} \norm f_1.\besov p(\mathbb R\times\cdots\times\mathbb R) . ,\qquad 0<p<\infty.
\end{equation}
If all the collections $\{\phi _{j,R}\} $ are uniformly adapted to $\mathcal R$, then the reverse inequality holds. 
 \end{theorem} 
 
Notice that  when neither collection of functions have zeros, the corresponding operator is \emph{not} bounded for general functions in $BMO $.  This is indicative of the well known fact that the result on Schatten norms is not as delicate as the criteria for being bounded.

Parallel to the above result, the dyadic version of Theorem \ref{t.two-schatten} and extension of Theorem \ref{c.schatten-one} is as follows.

\begin{theorem}
\label{t.dyadicMain} 
For any bounded $n$ parameter dyadic paraproduct  we have the equivalence 
\begin{equation}\label{e.dyadicMain}
\norm \operatorname B _{\textup{Haar}} (f_1,\cdot). \schatten p. \simeq \norm f_1.\besov p_{\text{\rm dyadic}}(\mathbb R\times\cdots\times\mathbb R). ,\qquad 0<p<\infty.
\end{equation}
\end{theorem}
 
Results similar to Theorem \ref{t.dyadicMain} above have appeared in the work of S. Pott and M. Smith, see \cite{MR2097606}.  
One of the aims for this paper is to provide the continuous version (i.e., Theorem \ref{t.two-schatten}), and provide a direct application of the paraproducts to the commutator of the multi-parameter Hilbert transforms, which links to the big Hankel operator \cite{CS}. For the sake of notational simplicity, we state and prove the result in the bi-parameter setting. The argument for multi-parameter setting follows similarly.

In Section~\ref{s.schatten} we collect some properties of Schatten norms. In Section~\ref{s.besov} we show that the Besov norms are independent of the choice of wavelet basis and establish the dyadic structure for the Besov spaces.    In Section \ref{s.one-dimensional} we give a proof of Theorems \ref{t.schatten-one} and \ref{c.schatten-one}.  We first will handle the case of Haar paraproducts since that will turn out to be a model for the more general case of paraproducts built from wavelet bases. In Section \ref{s.two-parameter} we give the proofs of Theorems \ref{t.two-schatten} and \ref{t.dyadicMain}, which are based upon the ideas appearing in Section \ref{s.one-dimensional} but will be complicated by additional notation necessary to handle the multi-parameter paraproducts.

 \section{Basic Properties of Schatten Norms}
  \label{s.schatten}

 Let $\mathcal H $ be a separable Hilbert space.   Recall that for elements $\varphi, \phi\in\mathcal H$ the operator denoted by $\varphi\otimes\phi$ takes an $f\in\mathcal H$ to $\phi\ip f,\varphi,_{\mathcal H} $.

 A compact operator $T \mid \mathcal H\mapsto\mathcal H $ has a decomposition 
 \begin{equation}  \label{e.singular value }
 T=\sum _n \lambda_n\, \operatorname e_n \otimes \operatorname f_n 
 \end{equation}
 in which $\lambda_n \in \mathbb R $, and $\{\operatorname e_n\} $ and $\{\operatorname f_n\} $ are 
 orthonormal sequences in $\mathcal H $ (compactness implies that $\abs{\lambda_n}\to0 $).
 The Schatten norm is 
 then 
 \begin{equation}  \label{e.schatten-def}
 \norm \operatorname T. \schatten p . {}\eqdef{}\left[ \sum _n\abs{\lambda_n }^p\right] ^{1/p},\qquad 0<p<\infty. 
 \end{equation}
 This is an actual norm for $1\le{}p<\infty $, while for $0<p<1 $ it is not.   The trace class operators are the class $\mathbb S_1 $ and the Hilbert--Schmidt operators are the class $\mathbb{S}_2$.  Part of the interest in these classes are that the class $\mathbb S_1 $ is in natural duality with $\mathcal L(\mathcal H) $, the space of bounded operators on $\mathcal H $ and the class $\mathbb{S}_2$ has a simple way to compute the norm using any orthonormal basis..  
It is clear that $\norm \operatorname T.\schatten p.=\norm \operatorname T^*.\schatten p. $.

 Define a collection of positive numbers 
  \begin{equation*}
  \mathcal T {}\eqdef{}\left\{ \left[ \sum_n \norm \operatorname  T \operatorname e_n .\mathcal{H}.^p \right]^{1/p} \mid \text{ $\{\operatorname e_n \}  $ is an orthonormal basis in $\mathcal H $ }\right\} . 
  \end{equation*}
  Then it is the case that 
  \begin{align}  \label{e.sch-inf} 
  \norm \operatorname  T.\schatten p .&{}=\inf\mathcal T,\qquad 0<p\le2,
  \\
 \label{e.sch-sup} 
  \norm \operatorname T.\schatten p .&{}=\sup\mathcal T,\qquad 2\le p<\infty.
  \end{align}

  For $1\le p<\infty $, the Schatten norm obeys the triangle inequality: 
 \begin{equation}  \label{e.schattentriangle}
 \norm \operatorname S+\operatorname T.\schatten p.\le{}\norm \operatorname S.\schatten p.+{} 
 		\norm \operatorname T.\schatten p..
 \end{equation}
 For $0<p<1 $ this is no longer the case.  There is the following quasi-triangle inequality, linked to the subadditivity of 
 $x\mapsto x^p $. 
 \begin{equation}  \label{e.schattentrianglep}
 \norm \operatorname S+\operatorname T.\schatten p.^p\le{}\norm \operatorname S.\schatten p.^p+{} 
 		\norm \operatorname T.\schatten p.^p.
 \end{equation}
In the converse direction, there is a proposition below.

\begin{proposition}
\label{p.ortholower}
Suppose that $\operatorname T $ is an operator from $\mathcal H $ to itself, and that $\operatorname P $ is a contraction, then 
\begin{equation*}
\norm \operatorname T \operatorname P .\schatten p.,\norm \operatorname P \operatorname T .\schatten p.{}\le{} \norm \operatorname T.\schatten p.,\qquad 0<p<\infty. 
\end{equation*}
\end{proposition} 
 
\begin{proof}  
For $\norm \operatorname P \operatorname T .\schatten p. $, this follows from the characterization of the Schatten norm in terms of either an infimum or supremum, see \eqref{e.sch-inf} and \eqref{e.sch-sup}.  Combining this observation with the equivalence of the Schatten norms for dual operators proves the proposition.
\end{proof}

 We also need  an inequality for the Schatten norms of a $m\times n $ matrix $\operatorname A=(a _{i,j }) $.  
 
\begin{proposition}
\label{p.mn}
We have the inequality 
\begin{equation}  
\label{e.mn}
\norm \operatorname A .\schatten p .\le{} (mn)^{\delta(p)}\left[\sum _{i,j=1}^{m,n} \abs{ a _{i,j } }^p\right]^{1/p},\qquad 0<p<\infty. 
\end{equation}
Here, $\delta(p)=\max\left(0,\frac12-\frac1p \right)$. 
\end{proposition}

 \begin{proof}   The case of $0<p<2 $ is clear. Appealing to  \eqref{e.sch-inf}, we use the standard 
 basis $\operatorname e_k $ $1\le{} k\le m $, so that 
 \begin{align*}
 \norm \operatorname A . \schatten p .^p  &  {}\le{} \sum _{k=1}^m \norm \operatorname A \operatorname e _k .. ^p 
 = {}  \sum_{k=1}^m \left[ \sum _{i=1}^n \abs { a _{i,k} }^2\right]^{p/2}
 {}\le{} \sum_{k=1}^m  \sum _{i=1}^n \abs { a _{i,k} }^p .
 \end{align*}
 This case is finished.

 To conclude the proof for $2<p<\infty $, observe that the norms decrease in $p $, hence 
 \begin{align*}
 \norm \operatorname A.\schatten p. &{}\le{} \norm \operatorname A . \schatten 2.
 ={} \left[ \sum _{i=1}^n  \sum _{j=1}^m \abs{ a _{i,j} }^2 \right]^{1/2}
 {}\le{} (mn) ^{\delta(p)} \left[ \sum _{i=1}^n  \sum _{j=1}^m \abs{ a _{i,j} }^p \right]^{1/p}.
 \end{align*}
 The proof is complete.
 \end{proof}
 
More comments about the Schatten norms and nearly weakly orthogonal (NWO) 
functions are made in Section 4.

 \smallskip
   \section{Besov space and its dyadic structure}
  \label{s.besov}

 \subsection{One-parameter}
  
We are interested in those results that relate the Schatten norms to Besov spaces of corresponding symbols.  The functions $\{\varphi_I\mid I\in\mathcal D \} $ will be a wavelet basis for $L^2(\mathbb R) $ with the function $\varphi$ being continuous and rapidly decreasing. 
In the first definition, \eqref{e.besov-one-def}, one may be concerned that the definition depends upon the choice of function $\varphi $.  There is a straight forward lemma which shows this is not the case. 

 \begin{proposition}\label{p.besov-equiv}
   
Let $\varphi $ and $\phi $ be two distinct wavelets, generating wavelet bases $\{\varphi_I\} $ and $\{\phi_I \} $, respectively.
We have the equivalence 
\begin{equation*}
\sum _{I\in\mathcal D } \left[ \abs I ^{-1/2}\abs{\ip f,\varphi_I , } \right] ^p\simeq{} \sum _{I\in\mathcal D } \left[ \abs I ^{-1/2}\abs{\ip f,\phi_I , } \right] ^p.
\end{equation*}
This is valid for any function $f $ for which either side is finite, and implied constants depend only on the 
choice of $0< p<\infty $. 
\end{proposition} 
 \begin{proof}
  This is a standard argument for wavelet characterisation of Besov spaces.
 We also note that the two  wavelet bases need not be associated with the same dyadic grid $\mathcal D$, which could be different grids.  The key steps are to use the wavelet expansion, almost orthogonality estimates and then the Plancherel--P\'olya type inequality. See for example the standard argument in \cite{HLW}.
\end{proof}
We have the standard characterization of the Besov space
$\besov p(\mathbb R) $, $1\leq p<\infty$, as follows (see for example \cite[p. 242]{Tri}), which also reflect that the definition of $\besov p(\mathbb R) $ is independent of the choice of the dyadic grids, and the associated wavelet basis.

 \begin{proposition}\label{classical B wavelet and B difference}

Let $b\in L_{\rm loc}^{1}(\mathbb R)$ and $0< p<\infty$. Then we say that $b$ belongs to the  Besov space $B^p(\mathbb R)$ if
$$ \|b\|_{B^{p}(\mathbb R )}=  \Bigg(\int_{ \mathbb R} \int_{\mathbb R}{\big|b(x)-b(y)\big|^p\over |x-y|^2}dydx \Bigg)^{1\over p}<\infty.  $$
Further, suppose $1\leq p<\infty$, then we have $B^p(\mathbb R) = \besov p(\mathbb R) $ with equivalence of norms.

\end{proposition}

 In the dyadic case the same definition applies, defining the Besov space $\besov p_{\text{\rm dyadic}}(\mathbb R)$ as follows, with the functions $\varphi_I $ replaced by Haar functions. 
 \begin{equation}  
 \norm f.\besov p _{\text{\rm dyadic}}(\mathbb R). {}\eqdef{}\left[ \sum _{I\in\mathcal D } \left[ \abs I ^{-1/2}\abs{\ip f,h_I , } \right] ^p \right] ^{1/p },
 \qquad 0<p<\infty.
 \end{equation}
 
 Based on our recent work \cite{LLW}, it is direct to see the following dyadic structure of the Besov space.
 
 \begin{proposition}\label{p.besov dy structure} Suppose $1\leq p<\infty$.
 There are two choices of dyadic grids $ \mathcal D^0$ and $ \mathcal D^1$ for which we have  
 $ \besov p _{\text{\rm dyadic},0}(\mathbb R)\cap  \besov p _{\text{\rm dyadic},1}(\mathbb R) = \besov p (\mathbb R)$, where $\besov p _{\text{\rm dyadic},i}(\mathbb R)$ is the dyadic Besov space associated with the dyadic grid $ \mathcal D^i$, $i=0,1$.
Moreover, we have 
 $$\|b\|_{\besov p(\mathbb R)}\approx   \|b\|_{\besov p _{\text{\rm dyadic},0}(\mathbb R)}+ \|b\|_{\besov p _{\text{\rm dyadic},1}(\mathbb R)}.$$
\end{proposition}

  \subsection{Product setting}

 Recall that the (product) Besov space on $\mathbb R^n=\mathbb R\times\cdots\times\mathbb R $ is defined by 
\begin{equation}
\norm f.\besov p(\mathbb R\times\cdots\times\mathbb R). {}\eqdef{}\left[ \sum _{R\in\mathcal R } \left[ \abs R^{-1/2}\abs{\ip f,\varphi_R , } \right] ^p \right] ^{1/p },
 \qquad 0<p<\infty.
 \end{equation}
 \begin{proposition}\label{p.besov-equiv}
   
Let $\varphi $ and $\phi $ be two distinct wavelets, generating wavelet bases $\{\varphi_R\} $ and $\{\phi_R \} $, respectively.
We have the equivalence 
\begin{equation*}
\sum _{R\in \mathcal R } \left[ \abs R ^{-1/2}\abs{\ip f,\varphi_R , } \right] ^p\simeq{} \sum _{R\in \mathcal R } \left[ \abs R ^{-1/2}\abs{\ip f,\phi_R , } \right] ^p.
\end{equation*}
This is valid for any function $f $ for which either side is finite, and implied constants depend only on the 
choice of $0< p<\infty $.  The two  wavelet bases need not be associated with the same dyadic grid $\mathcal R$.
\end{proposition} 
 \begin{proof}
Again, the key step is to use the wavelet expansion, almost orthogonality estimates, and then the Plancherel--P\'olya type inequality. See for example the standard argument for the product Hardy and BMO spaces in \cite[Section 4]{HLW}, which can be easily adapted to the product Besov space setting.
\end{proof}

We also introduce the following definition of product Besov space via difference.
For notational simplicity, for $j=1,\ldots,n$, we let 
$$ \triangle^{(j)}_{y_j} b(x_1,\ldots,x_n) = b(x_1,\ldots,x_n)-b(x_1,\ldots,x_{j-1}, y_j, x_{j+1},\ldots x_n).$$
 To begin with, we first introduce the following definition.
\begin{definition}
Let $b\in L_{\rm loc}^{1}(\mathbb R^n)$ and $0<p<\infty$. Then we say that $b$ belongs to the product Besov space $B^p(\mathbb R \times \cdots\times\mathbb R)$ if
$$ \|b\|_{B^{p}(\mathbb R \times\cdots\times \mathbb R)}=  \Bigg(\int_{ \mathbb R^2} \cdots\int_{\mathbb R^2}{\big| \triangle^{(1)}_{y_1}\cdots\triangle^{(n)}_{y_n} b(x_1,\ldots,x_n)\big|^p\over \prod_{j=1}^n|x_j-y_j|^2}dy_1dx_1\cdots dy_ndx_n \Bigg)^{1\over p}<\infty.  $$
\end{definition}
Next, we first point out that parallel to the classical setting, we have the equivalence of 
$ B^p(\mathbb R \times\cdots\times \mathbb R) $ and $\besov p(\mathbb R \times\cdots\times \mathbb R)$ when $p\geq1$.  That is,
 \begin{proposition}\label{p. B wavelet and B difference}
Suppose $1\leq p<\infty$. We have $B^p(\mathbb R \times\cdots\times \mathbb R) = \besov p(\mathbb R \times \cdots\times\mathbb R) $ with equivalence of norms.
\end{proposition}
\begin{proof}
 By using the reproducing formula and the almost orthogonality estimate in the tensor product setting (see for example \cite[Section 4]{HLW}) and Proposition \ref{p.besov-equiv}, one obtains the above proposition. Details are omitted.
\end{proof}

To establish the dyadic structure of the product Besov space, we only provide details for $\mathbb R^2=\mathbb R \times \mathbb R$. The $n$-parameter setting follows with appropriate modifications. 

We establish the following natural containment of the Besov space and the dyadic Besov space. 
 \begin{proposition}\label{p. B and Bd}
Suppose $1\leq p<\infty$. We have $B^{p}(\mathbb R \times \mathbb R)\subset \besov p_{\text{\rm dyadic}}(\mathbb R \times \mathbb R) $, with estimate
$$ \|b\|_{\besov p_{\text{\rm dyadic}}(\mathbb R \times \mathbb R)}\lesssim \|b\|_{B^{p}(\mathbb R \times \mathbb R)}.$$
\end{proposition}
\begin{proof}
The key inequalities here are specific to a choice of dyadic rectangle $R$. 
Let $R=I\times J$, where both $I$ and $J$ are dyadic intervals. 
Below, let $I' = I + 2\lvert  I \rvert$ and $J' = J + 2\lvert  J \rvert$.  Then it is clear that  $R'= I'\times J'$ is another dyadic rectangle with volume comparable to that of $R$. Hence, by the cancellation condition of $h_R$, we have
\begin{align*}
& \bigg|\int_Rb(x_1,x_2)h_R(x_1,x_2)dx_1dx_2\bigg|\, |R|^{-{1\over2}} \\
& \lesssim\inf_{(y_1,y_2)\in   R'}  \bigg|\int_R\Big(b(x_1,x_2)-b(x_1,y_2)-b(y_1,x_2)+b(y_1,y_2)\Big)h_R(x_1,x_2)dx_1dx_2\bigg|\, |R|^{-{1\over2}}
\\
& \lesssim |R|^{-{3\over2}}\int_{  R'}  \bigg|\int_R \triangle^{(1)}_{y_1}\triangle^{(2)}_{y_2} b(x_1,x_2)\ h_R(x_1,x_2)dx_1dx_2\bigg| 
\,  dy_1dy_2 
\\
& \lesssim |R|^{-{3\over2}} \bigg( \int_{  R'} \int_R\big| \triangle^{(1)}_{y_1}\triangle^{(2)}_{y_2} b(x_1,x_2)\ \big|^p dy_1dy_2dx_1dx_2 \bigg)^{1\over p}   \bigg(\int_{ R'} \int_R |h_R(x_1,x_2)|^{p'} dx_1dx_2dy_1dy_2\bigg)^{1\over p'}
\\
& \lesssim \bigg(\int_{ R'} \int_R \big| \triangle^{(1)}_{y_1}\triangle^{(2)}_{y_2} b(x_1,x_2)\ \big|^p dy_1dy_2dx_1dx_2 \ {1 \over |R|^2} \bigg) ^{\frac1 {p}} 
 \\
 & \lesssim \bigg(\int_{ R'} \int_R {\big| \triangle^{(1)}_{y_1}\triangle^{(2)}_{y_2} b(x_1,x_2)\ \big|^p\over |x_1-y_1|^2|x_2-y_2|^2}dy_1dy_2dx_1dx_2 \bigg) ^{\frac1p},
\end{align*}
where in the third inequality we use H\"older's inequality and this is where we need $1\leq p<\infty$. Note also that if $p=1$, then $p'=\infty$ and the second factor in the right-hand side of the third inequality will become $\| h_R \|_{L^\infty}$.

Take the power $p$ on both sides,  and sum over all $R, R'\in \mathcal D\times \mathcal D$, to get an expression dominated by $\lVert b \rVert_{B ^p  (\R\times \R)} ^p$.  
\end{proof}

Moreover, we also have a weaker version of the reverse containment. That is, we will need another set of dyadic intervals.  We take two dyadic grids $ \mathcal D^0 $ and $\mathcal D^1$ so that for all intervals $I$ there is a $Q\in 
 \mathcal D^0 \cup \mathcal D^1$ with 
 \begin{equation}
 \label{e:oneThird}   I\subset Q \subset 4 I . 
 \end{equation}
One option is that  $ \mathcal D^0 $ is the standard dyadic system in $\R$ and  $ \mathcal D^1 $ is the  `one-third shift' 
 of  $ \mathcal D^0 $, see for example \cite{MR3420475}.

  \begin{proposition}\label{p.besov-dyadic continuous} Suppose $1\leq p<\infty$.
   There are two choices of grids $ \mathcal D^0$ and $ \mathcal D^1$ for which we have  
 $$B^{p}(\mathbb R \times \mathbb R)= \mathbb B^{p,(0,0)} _{\text{\rm dyadic}}(\mathbb R \times \mathbb R)\cap  \mathbb B^{p,(0,1)} _{\text{\rm dyadic}}(\mathbb R \times \mathbb R)\cap  \mathbb B^{p,(1,0)} _{\text{\rm dyadic}}(\mathbb R \times \mathbb R)\cap  \mathbb B^{p,(1,1)} _{\text{\rm dyadic}}(\mathbb R \times \mathbb R) ,$$ where $\mathbb B^{p,(i,j)} _{\text{\rm dyadic}}(\mathbb R \times \mathbb R)$ is the dyadic Besov space associated with $ \mathcal D^i\times \mathcal D^j$ for $i,j\in\{0,1\}$.
Moreover, we have 
 $$\|b\|_{B^{p}(\mathbb R \times \mathbb R)}\approx   \|b\|_{\mathbb B^{p,(0,0)} _{\text{\rm dyadic}}(\mathbb R \times \mathbb R)}+ \|b\|_{\mathbb B^{p,(0,1)} _{\text{\rm dyadic}}(\mathbb R \times \mathbb R)}+ \|b\|_{\mathbb B^{p,(1,0)} _{\text{\rm dyadic}}(\mathbb R \times \mathbb R)}+ \|b\|_{\mathbb B^{p,(1,1)} _{\text{\rm dyadic}}(\mathbb R \times \mathbb R)}.$$

\end{proposition}

\begin{proof}
To begin with, we first note that
\begin{align}\label{e:IA}
 &\|b\|_{B^{p}(\mathbb R \times \mathbb R)}^p\\
 &=  \int_{ \mathbb R^2} \int_{\mathbb R^2} {\big| \triangle^{(1)}_{y_1}\triangle^{(2)}_{y_2} b(x_1,x_2)\big|^p\over |x_1-y_1|^2|x_2-y_2|^2}dy_1dy_2dx_1dx_2 
\nonumber\\
&\lesssim
 \sum_{\substack{I\in\mathcal D\\J\in\mathcal D}}\int_{I\times J} \int_{\substack{\{y_1\in \mathbb R:  2 ^{a} \lvert  I \rvert<|x_1-y_1|\leq 2 ^{a+1} \lvert  I \rvert\} \\ \{y_2\in \mathbb R:  2 ^{a} \lvert  J \rvert<|x_2-y_2|\leq 2 ^{a+1} \lvert  J \rvert\}}} {\big| \triangle^{(1)}_{y_1}\triangle^{(2)}_{y_2} b(x_1,x_2)\big|^p\over |x_1-y_1|^2|x_2-y_2|^2}dy_1dy_2dx_1dx_2 
\nonumber\\
&\lesssim
 \sum_{\substack{I\in\mathcal D\\J\in\mathcal D} }{1\over|R|^{2}}\int_{I\times J} \int_{\substack{\{y_1\in \mathbb R:  2 ^{a} \lvert  I \rvert<|x_1-y_1|\leq 2 ^{a+1} \lvert  I \rvert\} \\ \{y_2\in \mathbb R:  2 ^{a} \lvert  J \rvert<|x_2-y_2|\leq 2 ^{a+1} \lvert  J \rvert\}}} {\big| \triangle^{(1)}_{y_1}\triangle^{(2)}_{y_2} b(x_1,x_2)\big|^p}dy_1dy_2dx_1dx_2 
\nonumber
\\  
& \lesssim 
 \sum_{\substack{I\in\mathcal D\\J\in\mathcal D} }{1\over|R|^{2}} \sum _{m_1=n_1} ^{2n_1-1}\sum _{m_2=n_2} ^{2n_2-1}\nonumber\\
 &\qquad
 \int_{I\times J} \int_{\substack{\{y_1\in \mathbb R:  2 ^{a} {m_1\over n_1}\lvert  I \rvert<|x_1-y_1|\leq 2 ^{a} {m_1+1\over n_1} \lvert  I \rvert\}\\ \{y_2\in \mathbb R:  2 ^{a} {m_2\over n_2}\lvert  j \rvert<|x_2-y_2|\leq 2 ^{a} {m_2+1\over n_2} \lvert  J \rvert\}}} {\big| \triangle^{(1)}_{y_1}\triangle^{(2)}_{y_2} b(x_1,x_2)\big|^p}dy_1dy_2dx_1dx_2 .\nonumber
\end{align}
 The second integral above is over a symmetric interval. Consider the two intervals 
\begin{equation} \label{e:IAK}
I, \qquad    I + 2 ^{a}\lvert  I \rvert[m_1/n_1, (m_1+1)/n_1],  \qquad  n_1\leq m_1 < 2n_1.  
\end{equation}
Now, we  choose $a=5$, and $n_1 =1000$, so the second interval is smaller in length, but still comparable to $I$ in length.  And, they are separated by a distance approximately 
$2^a \lvert  I \rvert$. By \eqref{e:oneThird},  we can choose $a$ so that there is a dyadic $I'\in \mathcal D ^{0} \cup \mathcal D ^{1}$ 
which contains both intervals above, and moreover $I$ is contained in the left half of $I'$, and $I + 2 ^{a}\lvert  I \rvert[m_1/n_1, (m_1+1)/n_1]$ the right half. 
We can argue similarly for $I -2 ^{a}\lvert  I \rvert[m_1/n_1, (m_1+1)/n_1]$, as well as for the dyadic interval $J$ and the parameters $m_2$ and $n_2$. Below we continue with $R=I\times J$ and $R'=I'\times J'$.

Next, observe the following identity:
\begin{align*}
\triangle^{(1)}_{y_1}\triangle^{(2)}_{y_2} b(x_1,x_2)
&=\big(b(x_1,x_2) - E^{(1,0)}b(x_2)- E^{(0,1)}b(x_1)+ E^{(1,1)}b\big)\\
&\quad -\big( b(y_1,x_2) - E^{(1,0)}b(x_2)- E^{(0,1)}b(y_1)+ E^{(1,1)}b\big)\\
&\quad -\big( b(x_1,y_2) - E^{(1,0)}b(y_2)- E^{(0,1)}b(x_1)+ E^{(1,1)}b\big)\\
&\quad+ \big(b(y_1,y_2) - E^{(1,0)}b(y_2)- E^{(0,1)}b(y_1)+ E^{(1,1)}b\big)\\
&=:B_1(x_1,x_2)+B_2(y_1,x_2)+B_3(x_1,y_2)+B_4(y_1,y_2),
\end{align*}
where
$$E^{(1,0)}b(\cdot)={1\over |I'|}\int_{I'}b(z_1,\cdot)dz_1,\qquad 
E^{(0,1)}b(\cdot)={1\over |J'|}\int_{J'}b(\cdot,z_2)dz_2,$$
$$ E^{(1,1)}b={1\over |R'|}\int_{R'}b(z_1,z_2)dz_1dz_2. $$

In particular, for the main term in  \eqref{e:IA}, with fixed $n_1\leq m_1 < 2n_1$ and $n_2\leq m_2 < 2n_2$,  we have 
\begin{align}
&{1\over|R|^{2}} 
 \int_{I\times J} \int_{\substack{\{y_1\in \mathbb R:  2 ^{a} {m_1\over n_1}\lvert  I \rvert<|x_1-y_1|\leq 2 ^{a} {m_1+1\over n_1} \lvert  I \rvert\}\\ \{y_2\in \mathbb R:  2 ^{a} {m_2\over n_2}\lvert  j \rvert<|x_2-y_2|\leq 2 ^{a} {m_2+1\over n_2} \lvert  J \rvert\}}} {\big| \triangle^{(1)}_{y_1}\triangle^{(2)}_{y_2} b(x_1,x_2)\big|^p}dy_1dy_2dx_1dx_2\\
& \lesssim  
 {1\over|R|^{2}}  \Biggl\{ \int_{R'} \int_{R'}{\big|B_1(x_1,x_2)\big|^p}dy_1dy_2dx_1dx_2
 +
\int_{R'} \int_{R'}{\big|B_2(y_1,x_2)\big|^p}dy_1dy_2dx_1dx_2\nonumber\\
&\qquad\qquad+  \int_{R'} \int_{R'}{\big|B_3(x_1,y_2)\big|^p}dy_1dy_2dx_1dx_2
 +
\int_{R'} \int_{R'}{\big|B_4(y_1,y_2)\big|^p}dy_1dy_2dx_1dx_2
 \Biggr\}.\nonumber
\end{align}

It follows that the norm  $ \lVert b \rVert_{B ^p (\mathbb R \times \mathbb R)} ^p$ is dominated by several terms, one of which is 
\begin{equation} \label{e:1of8}
 \sum_{R\in\mathcal D^0\times \mathcal D^0} 
 \lvert  R\rvert^{-2} \int_{R} \int_{R}{\big|b(x_1,x_2) - E^{(1,0)}b(x_2)- E^{(0,1)}b(x_1)+ E^{(1,1)}b\big|^p}dy_1dy_2dx_1dx_2. 
\end{equation}
The other terms are obtained by varying the role of $m_1$ in \eqref{e:IAK} and the similar index $m_2$, considering the negative of the  intervals in \eqref{e:IAK}, 
exchanging the role of the dyadic grid, and  the roles of $x_i$ and $y_i$, $i=1,2$. All cases are similar, so we continue with the one above.  
In \eqref{e:1of8}, the point is that 
\begin{equation}
\mathbf 1 _{R} \big(b(x_1,x_2) - E^{(1,0)}b(x_2)- E^{(0,1)}b(x_1)+ E^{(1,1)}b\big) = \sum _{\substack{\tilde R \in \mathcal D ^{0}\times \mathcal D ^{0} \\  \tilde R\subset R}} \langle  b,h_{\tilde R} \rangle h_{\tilde R}. 
\end{equation}
That is, only the smaller scales contribute.  But then, it is straight forward to see that we can make a pure sum on scales.  
\begin{align}
\eqref{e:1of8} & \lesssim 
\sum_{R\in\mathcal D^0\times \mathcal D^0} 
 \lvert  
 R\rvert^{-2} \int_{R} \int_{R}{\big|\langle  b,h_R \rangle h_R\big|^p}dydx 
 \lesssim  \|b\|_{\mathbb B^{p,(0,0)} _{\text{\rm dyadic}}(\mathbb R \times \mathbb R)}^p. 
\end{align}
This completes the proof.  
 \end{proof}
Thus, based on Propositions \ref{p. B wavelet and B difference} and \ref{p.besov-dyadic continuous}, we obtain that

\begin{proposition}\label{p.besov-dyadic continuous 2}Suppose $1\leq p<\infty$.
   There are two choices of grids $ \mathcal D^0$ and $ \mathcal D^1$ for which we have  
 $$\mathbb B^{p}(\mathbb R \times \mathbb R)= \mathbb B^{p,(0,0)} _{\text{\rm dyadic}}(\mathbb R \times \mathbb R)\cap  \mathbb B^{p,(0,1)} _{\text{\rm dyadic}}(\mathbb R \times \mathbb R)\cap  \mathbb B^{p,(1,0)} _{\text{\rm dyadic}}(\mathbb R \times \mathbb R)\cap  \mathbb B^{p,(1,1)} _{\text{\rm dyadic}}(\mathbb R \times \mathbb R) ,$$ where $\mathbb B^{p,(i,j)} _{\text{\rm dyadic}}(\mathbb R \times \mathbb R)$ is the dyadic Besov space associated with $ \mathcal D^i\times \mathcal D^j$ for $i,j\in\{0,1\}$.
Moreover, we have 
 $$\|b\|_{B^{p}(\mathbb R \times \mathbb R)}\approx   \|b\|_{\mathbb B^{p,(0,0)} _{\text{\rm dyadic}}(\mathbb R \times \mathbb R)}+ \|b\|_{\mathbb B^{p,(0,1)} _{\text{\rm dyadic}}(\mathbb R \times \mathbb R)}+ \|b\|_{\mathbb B^{p,(1,0)} _{\text{\rm dyadic}}(\mathbb R \times \mathbb R)}+ \|b\|_{\mathbb B^{p,(1,1)} _{\text{\rm dyadic}}(\mathbb R \times \mathbb R)}.$$

\end{proposition}

\bigskip
 
 \section{The Proof of the One Parameter Result}  
 \label{s.one-dimensional}


 The proof in fact has very little to do with the function theory of Besov spaces.  The main results in Theorem \ref{t.schatten-one} and Theorem \ref{c.schatten-one} can be rephrased this way. 
 
\begin{theorem}
\label{t.withoutbesov}
Suppose that $\{\varphi_I\mid I\in\mathcal D\} $ and $\{\phi_I\mid I\in\mathcal D \} $ are adapted to the dyadic intervals, and at least one collection has zeros,  then we have the estimate 
\begin{equation} 
\label{e.withoutbesov}
\norm \sum _{I\in\mathcal D} \alpha(I)\, \varphi_I \otimes \phi_I .\schatten p. {}\lesssim{} \norm \alpha(\cdot).\ell^p ., \qquad 0<p<\infty. 
\end{equation} 
If both collections of functions are uniformly adapted to $\mathcal D $, then the reverse inequality holds. 

In the setting of the Haar functions we have the estimate 
\begin{equation} 
\label{e.Haarwithoutbesov}
\norm \sum _{I\in\mathcal D} \alpha(I)\, h^\epsilon_I \otimes h^\delta_I .\schatten p.\simeq{} \norm\alpha(\cdot).\ell^p ., \qquad 0<p<\infty,  
\end{equation} 
with $\{\epsilon,\delta\}\not=\{1,1\} $.
\end{theorem} 
 
We first point out that the conditions of $\varphi_I$ and $\phi_I$ ``to be adapted to the dyadic intervals'' are close to the notion of ``nearly weakly orthogonal sequences'' introduced by Rochberg and Semmes, \cite{RochSem}. However, we do not require the compact support condition, but require suitable decay and regularity instead. Thus, it has the advantage in dealing with the continuous setting via wavelet functions or more general functions, which was missing before.

Of course the Haar functions are an orthonormal basis for  $L^2(\mathbb R) $, and the first and most elementary example of a wavelet.  The proof below simplifies considerably for Haar functions, and the reader is encouraged to read through the argument with this case in mind first.  
 
 \subsection{The Proof of Theorem~\ref{t.withoutbesov}: The Haar Case.} 
 
 We begin with the obvious estimate 
 \begin{equation*}
 \norm \sum _{I\in\mathcal D}  \alpha(I) \, h_I^0 \otimes h_I^0 . \schatten p .  {}\simeq{} \norm \alpha(\cdot) .\ell^p .  .
 \end{equation*}
To prove the full argument in \eqref{e.Haarwithoutbesov}, a certain extension of the above fact is needed. Without lost of generality we will prove the following: 
 \begin{equation}\label{upper 4.3}
\norm \sum _{I\in\mathcal D}  \alpha(I) \, h_I^1 \otimes h_I^0 . \schatten p .\lesssim {} \norm \alpha(\cdot) .\ell^p . 
 \end{equation}
and
 \begin{equation} \label{lower 4.3}
\norm \sum _{I\in\mathcal D}  \alpha(I) \, h_I^1 \otimes h_I^0 . \schatten p .\gtrsim {} \norm \alpha(\cdot) .\ell^p .. 
 \end{equation}

 The main point is the explicit representation 
 \begin{equation}  \label{e.expandindicator} 
  h^1_I{}=\sum _{J:\ I\varsubsetneqq J }{\sqrt{\abs I}} { h_J(c(I)) } \, h_J^0 
  {}=\sum _{J:\ I\varsubsetneqq J } \nu_{I,J} \sqrt{\tfrac{\displaystyle \abs I}{\displaystyle \abs J} }   \, h_J^0 ,
  \end{equation}
  where the $\nu_{I,J} \in\{\pm1\} $ are determined by $I $ being in the left or the right half of $J $.

 For an integer $m>0 $ and for $J\in\mathcal D$, we define the $m$-fold children of $J$ to be
 \begin{equation}  \label{e.dyadicJ}
 \mathcal D(m,J) {}\eqdef{}\{ I\in\mathcal D\mid I\subset J, 2^m\abs I=\abs J\}.
 \end{equation}
 Now, consider the operators 
 \begin{equation} \label{e.Hmdef}
  H_{m,J} {}\eqdef{}\sum _{I\in\mathcal D(m,J)}  \nu_{I,J}\,\alpha(I)\, h_I^0
 \end{equation}
and
 \begin{equation} \label{e.Smdef}
 \operatorname S_m {}\eqdef{}  \sum _{J\in\mathcal D} h_J^0\otimes  H_{m,J}.
 \end{equation}
 The choices of signs $\nu_{I,J} $ are determined as in \eqref{e.expandindicator}. 
 Our observation is that this operator $\operatorname S_m$ also has an effective estimate of its norm as follows.  
 \begin{equation}  \label{e.Sm<}
 \norm \operatorname S_m.\schatten p. {}\lesssim{} 2 ^{\delta(p)m} \norm \alpha(\cdot).\ell^p., \qquad 0<p<\infty, 
 \end{equation}
 where $\delta(p)=\max\left(0,\frac12-\frac1p\right) $.
 Indeed, the functions $H_{m,J} $ are orthogonal in $J $, and  we have the estimate 
 \begin{equation} 
 \label{e.HmJ}
 \norm H _{m,J }.2.{}={} \left[ \sum _{I\in\mathcal D(m,J)} \abs { \alpha(I) }^2\right]^{1/2}
{} \le{} 2 ^{ \delta(p)m}\left[ \sum _{I\in\mathcal D(m,J)} \abs { \alpha(I) }^p\right]^{1/p}.
 \end{equation}
   And this clearly implies our observation in \eqref{e.Sm<}. 
 
 \medskip 
 
 Our inequality \eqref{e.Haarwithoutbesov} then follows.  Taking \eqref{e.expandindicator} and 
 \eqref{e.Sm<} into account, it is clear that we can write 
 \begin{equation*}
 \sum _{I\in\mathcal D} \alpha(I)\, h^1_I\otimes h^{0}_I {}={}\sum _{m=1}^\infty 2^{-m/2} \operatorname S_m .
 \end{equation*}
 For $1\le{}p<\infty $, the Schatten norm obeys the triangle inequality, hence, together with \eqref{e.Sm<}, we can estimate 
 \begin{equation} \label{e.normdecay}
 \begin{split}
 \norm \sum _{I\in\mathcal D} \alpha(I) h^1_I\otimes h^{0}_I . \schatten p.&{}\le{}
 {} \sum _{m=1}^\infty 2^{-m/2} \norm \operatorname S_m .\schatten p . 
 \\ & \le{} \norm \alpha(\cdot).\ell^p . \sum _{m=1}^\infty 2^{-m \min\left(\frac12,\frac1p\right)} 
 \\ & {}\lesssim{}\norm \alpha(\cdot).\ell^p . .
 \end{split}
 \end{equation}
 In the case that $0<p<1 $, we can rely upon the subadditivity as in \eqref{e.schattentrianglep}, and a very similar 
 argument finishes the proof of the upper bound as in \eqref{upper 4.3}.

 \bigskip 
 
 We  prove the lower bound as in \eqref{lower 4.3}.  Fix $\alpha (\cdot) $ so that $\norm \alpha(\cdot) .\ell ^p .=1 $, 
 We want to show that for 
 \begin{equation*}
\operatorname T_\alpha {}\eqdef{}\sum _{I\in\mathcal D} \alpha(I)\, h^1_I\otimes h_I^0,
 \end{equation*}
we have $\norm \operatorname T_\alpha.\schatten p. {}\gtrsim{}1 $.

 A collection of dyadic intervals 
 $\mathcal D_\ell $ will \emph{have scales separated by $\ell $ } if 
  it satisfies the conditions 
  $I\in \mathcal D_\ell  $ implies $I\pm2\abs I \in \mathcal D_\ell $, but $I\pm\abs I \not\in \mathcal D_\ell $ and  $\{\log_2 \abs 
  I\mid I\in\mathcal D_\ell\}  {}={} a+\ell \mathbb Z $ for some choice of integer $a $. 
 All dyadic intervals $\mathcal D $ are a union of $2\ell $ subcollections with scales separated by $\ell $.  
 Of course passing to such a subcollection will suggest a loss of order $\ell^{-1} $, but from other 
 aspects of the argument below, we will be able to pick up an exponential decay in $\ell $.

 For an integer $\ell $ to be chosen,  we can separate the scales in the 
 dyadic intervals, choosing one particular $\mathcal D_\ell $ so that 
 \begin{equation}  \label{e.loweralpha} 
  \sum _{I\in\mathcal D_\ell } \abs{\alpha(I)}^p\ge(2\ell)^{-p}. 
  \end{equation} 
 Let $\operatorname P_\ell $ be the projection onto the span of the functions $\{ h_J\mid J\in\mathcal D_\ell \} $, i.e.,
 $$ \operatorname P_\ell (f) = \sum_{J\in D_\ell} \langle f,h_J^0\rangle h_J^0. $$
Let $\mathcal D_\ell' $ be the collection of \emph{parents}  of  those dyadic intervals in $\mathcal D_\ell $  (the parent of $I $ is the next  dyadic larger interval) and let $\operatorname P'_\ell $ be the corresponding projection.  
 Appealing to Proposition~\ref{p.ortholower} we  have the estimate 
 \begin{equation*}
 \norm \operatorname T_\alpha.\schatten p.\ge\norm \operatorname P_\ell' \operatorname T_\alpha \operatorname P_\ell
 .\schatten p.  .
 \end{equation*}
 We shall show that  $\norm \operatorname P_\ell' \operatorname T_\alpha \operatorname P_\ell
 .\schatten p.$ is at least $(8 \ell)^{-1} $  for $\ell $ sufficiently large 
 depending only on $p $. 
 
 Define
 \begin{equation*}
 \operatorname T^1_\alpha {}\eqdef{} \sum _{I\in\mathcal D_\ell}  v_{I,I'} \ \alpha(I)\, h _{I' }^0\otimes h_I^0, 
 \end{equation*}
 where $I' $ is the parent of $I $ and again the $\nu_{I,I'} \in\{\pm1\} $ are determined by $I $ being in the left or the right half of $I' $. 
 Then the $\mathbb S^p$ norm can be calculated exactly from \eqref{e.schatten-def} as follows
 \begin{equation*}
 \norm \operatorname  T^1_\alpha.\schatten p.^p=\sum _{I\in\mathcal D_\ell}   \alpha(I)^p {}\ge{}(2\ell)^{-p},
 \end{equation*}
 where the above inequality follows from \eqref{e.loweralpha}.
 This is the main term in providing a lower bound on the Schatten norm for $\operatorname T_\alpha $.  
 
 We appeal to some of the estimates used in the proof of the upper bound.  Define 
 \begin{equation*}
\operatorname  T^m_\alpha {}\eqdef{} \sum _{I\in\mathcal D'_\ell}  h_I\otimes \widetilde H _{m\ell+1, I},
 \end{equation*}
where $\widetilde H_{m\ell+1, I}$ is similar as in the definition from \eqref{e.Smdef}, given by
 \begin{equation} \label{e.tilde Hmdef}
  \widetilde H_{m\ell+1, I} {}\eqdef{}\sum _{\substack{ K\in \mathcal D_\ell\\ K\in\mathcal D(m\ell+1,I)}}  \nu_{K,I}\,\alpha(K)\, h_K^0.
 \end{equation}
 
 Note that 
 \begin{equation*}
 \operatorname P_\ell' \operatorname T_\alpha \operatorname P_\ell=2^{-{1\over2}} T^1_\alpha+ \sum _{m=2}^\infty 2 ^{-\frac12(m\ell+1)}\operatorname T^m_\alpha .
 \end{equation*}
Repeating the estimates as in \eqref{e.Sm<} and \eqref{e.normdecay}, we have the estimate 
 \begin{equation*}
 \norm \sum _{m=2}^\infty 2 ^{-\frac12(m\ell+1)}\operatorname T^m_\alpha .\schatten p. {}\lesssim{}2^{-\min\left(\frac12,\frac1p\right) \ell } .
 \end{equation*}
 The implied constant depends upon $0<p<\infty $.  Therefore, for an absolute choice of $\ell $, we will have the estimate 
 $\norm \operatorname T_\alpha .\schatten p . \ge(8\ell)^{-1} $.  This completes the proof of \eqref{lower 4.3}.
 
Thus,  combining \eqref{upper 4.3} and \eqref{lower 4.3}, we obtain that \eqref{e.Haarwithoutbesov} holds.  The proof  of the Haar Case in Theorem~\ref{t.withoutbesov} is complete.
\hskip5.5cm $\Box$

\subsection{The Proof of Theorem~\ref{t.withoutbesov}: The Wavelet Case} 
  
Assume that both  collections $\{\varphi_I \} $ and $\{\phi_I\} $  are wavelet bases, then the operator in \eqref{e.withoutbesov} is already given in singular value  form, and the theorem is trivial.  The Theorem follows from the lemma below:

\begin{lemma}
\label{l.gettingridofzeros}
We have the inequality 
\begin{equation} 
\label{e.gettingridofzeros} 
\norm \sum _{I\in\mathcal D} \alpha(I) \varphi_I\otimes \phi_I .\schatten p. {}\lesssim{}\norm \alpha(\cdot) .\ell^p.,\qquad 0<p<\infty,
\end{equation}
assuming only that $\{\phi_I \} $ are adapted to $\mathcal D $.  If this collection of functions is uniformly adapted to $\mathcal D $, then the reverse inequality holds.  
\end{lemma}

  In the case of Haar functions, the main point is the explicit expansion of \eqref{e.expandindicator}. 
  In the current setting, of course we have the expansion 
  \begin{equation*}
  \varphi_I=\sum _{J\in\mathcal D} \ip \varphi_I,w_J,\, w_J ,
  \end{equation*}
  as $\{w_J \} $ is an orthogonal basis.  But the expansion is not quite so clean.  Nevertheless, we have the 
  following general almost orthogonality estimate.

  \begin{lemma}\label{l.coefficients}
  Denoting $2^m\abs I=\abs J $, we have the inequality 
  \begin{equation}\label{e.wIJ} 
 \abs{ \ip \varphi_{I},w_J, } {}\lesssim{} 2^{- \Delta (m)}\left( 1+\frac{\text{\rm dist}(I,J)}{\abs I+\abs J }\right) ^{-\eta}
	 , \qquad   \ m\in \mathbb Z.
 \end{equation}
 Here, $\Delta(m)={} \abs m $ if $m\le{}0 $ , and $\Delta(m)=\frac12m $ if $m>0 $, $\eta>0 $ is a large positive 
 constant, namely $N-1 $, where $N $ appears in \eqref{e.adapted}. 
 \end{lemma}

\begin{proof}  This is elementary.  
 On the one hand, by using \eqref{e.adapted}, we have 
 \begin{equation*}
 \abs{ \ip \varphi_{I},w_J, } {}\lesssim{} 
 2^{-\frac12 \abs  m}\left( 1+\frac{\text{dist}(I,J)}{\abs I+\abs J}\right) ^{-N+1},
 \end{equation*}
where $N $ is as in \eqref{e.adapted}.   This treats the case $m>0 $.  

The case $m<0 $ does not occur in the Haar setting.  While it does occur here, there is an 
extra decay coming from the fact that the 
wavelet has mean zero,  and are adapted to an interval of smaller length than $J $.  Thus,  we essentially 
gain a derivative in this case
 \begin{align*} 
 \abs{ \ip \varphi_{I},w_J, } {}\lesssim{}2^{-\frac32\abs m}. 
 \end{align*}
 Taking the geometric mean of these two estimates proves the estimate in this last case. 
 \end{proof}

 We use this lemma to prove another technical lemma, more specifically adapted to our purposes.  
 Let $\operatorname P _{2^m} $ be the wavelet projection of functions onto the span of $\{w _{J}\mid \abs J=2^m\} $. 
 that is 
 \begin{equation*}
 \operatorname P_{2^m} {}\eqdef{} \sum_{J\in\mathcal D:\ \abs J=2^m} w_J\otimes w_J,
 \end{equation*}
 so that 
 \begin{equation*}
 \operatorname P_{2^m}(f) = \sum_{J\in\mathcal D:\ \abs J=2^m}  \langle f , w_J\rangle w_J.
 \end{equation*}
 
 \begin{lemma}\label{l.movingprojection}
 We have the estimate 
 \begin{align*}
 \norm \operatorname B_m.\schatten p. &{}\lesssim{}2 ^{-\min(\frac12,\frac1p)\abs m}\norm \alpha(\cdot). \ell^p .,\qquad m\in \mathbb Z,
 \\[6pt]
 \text{where}\quad  \operatorname B_m &{}\eqdef{} \sum _{I\in\mathcal D }
 \alpha(I) [\operatorname P_{2^m \abs I}\varphi_I] \otimes  \phi_I . 
 \end{align*}
 In particular note that the scale of the wavelet projection being used depends upon the scale of $ I $. 
 \end{lemma}

 \begin{proof}
 Consider first the case of $m<0 $. 
 Now, for an integer $\ell \ge0$, set 
 \begin{align*}
 \mathcal D_0(I)& {}\eqdef{}\{ J\in\mathcal D\mid \abs J=2^m\abs I,\ J\subset I\},\\ 
 \mathcal D_\ell(I) &{}\eqdef{}\{ J\in\mathcal D\mid \abs J=2^m\abs I,\ J\subset  2^\ell I,\ J\not\subset2 ^{\ell-1 }I\}. 
 \end{align*}
 Also consider the functions  
 \begin{equation*}
 W _{I,\ell} {}\eqdef{}\sum _{J\in\mathcal D_\ell(I) } \alpha(I) \ip \varphi _{I},w_J, w_J .
 \end{equation*}
 For fixed $\ell $, the functions $\{W _{I,\ell }\mid I\in\mathcal D \} $ are orthogonal.  Hence the operator 
 \begin{equation*}
 \operatorname T_\ell  {}\eqdef{}\sum _{I\in\mathcal D }    W _{I,\ell}\otimes \phi_I  
 \end{equation*}
 is in singular value form.   Moreover, observe that from Lemma \ref{l.coefficients},
 \begin{align*}
 \norm W _{I,\ell}.2.&{}\le{} \left[ \sum _{J\in\mathcal D_\ell (I)} \abs {\alpha(I) \ip \varphi _{I},w_J, }^2\right] ^{1/2 }
\lesssim
 2^{-|m|-\eta'\ell }  \left[ \sum _{J\in\mathcal D_\ell(I) } \abs {\alpha(I) }^2 \right] ^{1/2 }. 
 \end{align*}
 An elementary estimate gives 
\begin{equation}  
\label{e.l2lp}
\left[ \sum _{J\in\mathcal D_\ell(I) } \abs {\alpha(I) }^2 \right] ^{1/2 } {}\lesssim{}\text{card}(\mathcal D_\ell(I))^{\delta(p)} \left[ \sum _{J\in\mathcal D_\ell(I) } \abs {\alpha(I) }^p \right] ^{1/p }    .
\end{equation}
 The term in the exponent is $\delta(p)=\max\left(0,\frac12-\frac1p\right) $ as before.

 Hence, we can easily estimate the Schatten norm of $\operatorname T_\ell $. 
 \begin{align*}
 \norm \operatorname T _\ell .\schatten p.^p 
 {}\le{} \sum _{I\in\mathcal D}  \norm W _{I,\ell }.2.^p 
 {}\lesssim{} 2^{-p(1-\delta(p))|m|-\eta''\ell } \norm  \alpha(\cdot) .\ell^p .^p . 
 \end{align*}
 Because we are free to take $\eta''{}$ as large as needed, this completes the proof in this case.

 \medskip 
 
 We now consider the case of $m>0 $.  
 Keeping the same notation $\mathcal D_\ell(J) $, we redefine 
 \begin{align*}
 W _{J,\ell} &{}\eqdef{}\sum _{I\in\mathcal D_\ell(J)}  \alpha(I)\ip \varphi_I,w_J , w_J, 
 \\ 
 \operatorname T _\ell & {}\eqdef{} \sum _{J\in\mathcal D }w_J\otimes  W _{J,\ell } .
 \end{align*}
 Again, this is an operator in singular value  form.  In particular 
 \begin{align*}
 \norm W _{J,\ell }.2.^2&{}={} 
 		\sum _{I\in\mathcal D_\ell(J)} \abs{\alpha(I)\ip \varphi_I,w_J ,}^2
 \\& {}\lesssim{} 2 ^{-m -\eta \ell }\sum _{I\in\mathcal D(J)} \abs{\alpha(I)}^2
 \\& {}\lesssim{} 2 ^{-m(1-2 \delta(p))-\eta \ell }\Biggl[\sum _{I\in\mathcal D(J)} \abs{\alpha(I)}^p
  \Biggr] ^{2/p}.
 \end{align*}
 Therefore, it is the case 
 \begin{equation*}
 \norm \operatorname T _{\ell }.\schatten p .^p  
 {}\lesssim{} 2 ^{-mp\min\left(\frac12,\frac1p\right)-\eta \ell} \norm \alpha(\cdot).\ell^p.^{p}.
\end{equation*}
 Due to the fact that $\eta $ can be taken very large, for all $0<p<\infty $, 
 this estimate can be summed over $\ell$, to prove the lemma in this case.  The proof of Lemma \ref{l.movingprojection} is complete.
 \end{proof} 
 
 \begin{proof}[Proof of Lemma~\ref{l.gettingridofzeros}.]  
 We assume that $\{\phi_I \} $ are adapted to $\mathcal D $.  Then, we have 
 \begin{align*}
 \sum _{I\in\mathcal D}\alpha(I)\,\varphi_I\otimes \phi_I {}& =  \sum _{I\in\mathcal D}\alpha(I)\, \sum _{J\in\mathcal D} \ip \varphi_I,w_J,\, w_J  \otimes \phi_I 
\\
& = \sum_{m\in\mathbb Z} \sum _{I\in\mathcal D}\alpha(I)\, \sum _{J\in\mathcal D:\ \abs J=2^m} \ip \varphi_I,w_J,\, w_J  \otimes \phi_I  \\
& = \sum_{m\in\mathbb Z} \sum _{I\in\mathcal D }
 \alpha(I) \big[\operatorname P_{2^m \abs I}\varphi_I\big] \otimes  \phi_I
   \\
& ={}  \sum _{m\in \mathbb Z}\operatorname B_m.
 \end{align*}

 With the estimates on the operators $\operatorname B_m$ provided by Lemma~\ref{l.movingprojection}, we obtain that 
 $$\norm \sum _{I\in\mathcal D} \alpha(I) \varphi_I\otimes \phi_I .\schatten p. {}\lesssim{}\norm \alpha(\cdot) .\ell^p.
$$ 
for $0<p<\infty$, which shows that \eqref{e.gettingridofzeros}  holds.
  
 \medskip 
  We turn to the proof of the reverse inequality, assuming that $\{\phi_I\} $ are uniformly adapted to $\mathcal D $. We aim to prove that 
  \begin{equation} 
\label{e.gettingridofzeros reverse} 
\norm \sum _{I\in\mathcal D} \alpha(I) \varphi_I\otimes \phi_I .\schatten p. {}\gtrsim{}\norm \alpha(\cdot) .\ell^p.,\qquad 0<p<\infty.
\end{equation}

  This argument is modeled on the proof of the lower bound in the Haar setting. 
  Recall that this means that \eqref{e.uniformlyadapted} is in force. 
  Fix a dyadic interval $I_0 $ with length 1, so that 
  \begin{equation*}
  \abs{\ip \phi _{[0,1]},w _{I_0}, } 
  \end{equation*}
  is maximal.  Write $I<J $ if the (orientation preserving) linear transformation that carries $J $ to $[0,1] $ 
  also carries $I $ to $I_0$. (In the Haar case, $I_0 $ was the parent of $[0,1] $.) 
  
  The notion of \emph{scales separated by $\ell$} is modified slightly. 
 A collection of dyadic intervals 
 $\mathcal D_\ell $ will \emph{have scales separated by $\ell $ } if 
  it satisfies the conditions:  ``$I\in \mathcal D_\ell  $ implies $I\pm\ell\abs I \in \mathcal D_\ell $, but $I\pm j\abs I \not\in \mathcal D_\ell $ 
  for $\abs j<\ell $ and  $\{\log_2 \abs 
  I\mid I\in\mathcal D_\ell\}  {}={} a+\ell \mathbb Z $ for some choice of integer $a $''. 
  
 All dyadic intervals $\mathcal D $ are a union of $\ell ^2 $ subcollections with scales separated by $\ell $.  
  We expect a loss of $\ell ^{-2} $ by passing to a subcollection with scales separated by $\ell $.  We will 
  be able to pick up rapid (but not exponential) decay from other parts of the argument. 
 
 Fix $\alpha(\cdot) $ such that $\norm \alpha(\cdot).p.=1 $. 
 We want to provide a lower bound on 
 \begin{equation*}
 \operatorname T_{\alpha}\eqdef{}\sum _{I\in\mathcal D}\alpha(I)\, \phi_I\otimes w_I. 
 \end{equation*}
 For a choice of $\ell$ to be specified, we can choose $\mathcal D _{\ell}$ with scales separated by $\ell$ such that 
 \begin{equation*}
\bigg( \sum _{I\in\mathcal D_\ell}\abs{\alpha(I)}^p\bigg)^{1\over p}\ge{} \ell ^{-2}.
 \end{equation*}
  Let $\operatorname P_\ell $ be the projection onto the span of $\{w _{I}\mid I\in\mathcal D_\ell \} $, that is 
 \begin{equation*}
 \operatorname P_{\ell} {}\eqdef{} \sum_{ I\in \mathcal D_\ell} w_I\otimes w_I . 
 \end{equation*}
  Let $\mathcal D_{\ell,<} $ be those dyadic intervals $I $ which satisfy $I<J $ for some $J\in\mathcal D_\ell $.  
  We will pick $\ell $ so large that there is a unique such $J $. 
  Let $\operatorname P_{<,\ell} $ be the corresponding wavelet projection onto the span of $\{w _{I_<}\mid I_<\in\mathcal D_{\ell,<}\} $, defined similarly to $P_\ell$ as above.
  
By Proposition~\ref{p.ortholower}, we have the estimate 
\begin{equation*}
\norm \operatorname T_{\alpha}.\schatten p.\ge{} \norm  \operatorname P_{<,\ell} \operatorname T_{\alpha}\operatorname P_\ell .\schatten p. .
\end{equation*}
We estimate the norm of the latter quantity.   By definition, we have
$$\operatorname P_{<,\ell} \operatorname T_{\alpha}\operatorname P_\ell = 
\sum_{ I\in \mathcal D_{\ell,<}} \sum_{ I'\in \mathcal D_{\ell} }\alpha(I')
\ip \phi_{I'},w _{I_<},\, w _{I_<}\otimes w_{I'}.  $$
 
  Set
  \begin{equation*}
  \operatorname T_{\alpha}^1\eqdef{} \sum_{I\in\mathcal D_\ell} \alpha(I) \ip \phi_I,w _{I_<},\, w _{I_<}\otimes w_I .
  \end{equation*}
  Here, by $I_<$ we mean that unique element of $\mathcal D_{\ell,<} $ for which $I_<<I $.  Observe that 
  $\ip \phi_I,w _{I_<}, $ is in fact independent of $I$, and so we have the estimate 
  \begin{equation}\label{essential term large}
  \norm \operatorname T_{\alpha}^1 .\schatten p.{}\gtrsim{}\ell ^{-2},
  \end{equation}
  with the implied constant depending only on the specific choice of $\phi$ and wavelet basis $\{w_I\} $. 
  
  The remainder of the argument consists of showing that 
  \begin{equation}\label{remainder small}
  \norm  \operatorname P_{<,\ell} \operatorname T_{\alpha}\operatorname P_\ell- \operatorname T^1_{\alpha}.\schatten p. {}\lesssim{} \ell ^{-4},
  \end{equation}
  where the implied constant depends on $0<p<\infty $, and on the specific constants that enter into 
  the inequality \eqref{e.adapted}  (in particular, for the case of $0<p<1 $, we will need to require that 
  $Np>5 $).  
    
 To see  \eqref{remainder small},   note that the remainder $ \operatorname P_{<,\ell} \operatorname T_{\alpha}\operatorname P_\ell- \operatorname T^1_{\alpha}$ can be represented and split as follows.
Set 
  \begin{align*}
 \mathcal D_{\ell,j}(I') &{}\eqdef{}\{ I\in\mathcal D_{\ell,<}\mid \abs  I\subset  2^j I',\ I\not\subset2 ^{j-1 }I'\},\quad j>1. 
 \end{align*}
 And we  consider the functions  
 \begin{equation*}
 W _{I',j} {}\eqdef{}\sum _{I\in\mathcal D_{\ell,j}(I') } \alpha(I') \ip \phi_{I'},w _{I_<},\, w _{I_<}.
 \end{equation*}
Then we have
 \begin{align*}
 \operatorname P_{<,\ell} \operatorname T_{\alpha}\operatorname P_\ell- \operatorname T^1_{\alpha}&= 
\  \sum_{ I'\in \mathcal D_{\ell} } \sum_{ I\in \mathcal D_{\ell,<},  I\not=I'} \alpha(I')
\ip \phi_{I'},w _{I_<},\, w _{I_<}\otimes w_{I'}\\
&=\sum_{j=2}^\infty\sum_{ I'\in \mathcal D_{\ell} }   W _{I',j}\otimes w_{I'}\\
&\hskip-.1cm \eqdef\sum_{j=2}^\infty T_{\ell,j}.
  \end{align*}
  Note that for each $j$,  there are only $2j$ cases of $I\in \mathcal D_{\ell,j}(I')$.  Hence, by using the almost orthogonality estimate in Lemma \ref{l.coefficients} for $\ip \phi_{I'},w _{I_<},$, and following a similar step in the proof of Lemma \ref{l.movingprojection} for $ \|W _{I',j}\|_2 $, we obtain that 
$\norm  T_{\ell,j}.\schatten p.\lesssim 2^{-\eta j}2^{-\eta \ell}$, where $\eta$ is a large positive number as in Lemma \ref{l.coefficients}. Thus, summing over all $j\geq2$, we obtain that 
$  \norm  \operatorname P_{<,\ell} \operatorname T_{\alpha}\operatorname P_\ell- \operatorname T^1_{\alpha}.\schatten p. {}\lesssim{} 2 ^{-\eta \ell},
$
that is,  \eqref{remainder small} holds.

  Now, it is clear that we can choose $\ell$ sufficiently large, and then combine \eqref{essential term large} with \eqref{remainder small} to obtain that 
  \eqref{e.gettingridofzeros reverse}  holds.  
The proof  of the Wavelet Case in Theorem~\ref{t.withoutbesov} is complete.
 \end{proof}

 \section{The Proof of the Multi-Parameter Result}
 \label{s.two-parameter}
 
\subsection{Two-Parameter Paraproducts}
We now consider paraproducts formed over sums of dyadic rectangles in the plane.  The class of paraproducts is then invariant under a two parameter family of dilations, a situation that we refer to as one of ``two parameters''.   This case contains all the essential difficulties for the higher parameter setting and is good for focusing ideas.  Again, there is very little function theory in the argument and Theorems \ref{t.two-schatten} and \ref{t.dyadicMain} in the two parameter setting can be rephrased as

 
 \begin{theorem}
 \label{t.two-withoutbesov}
Let $\{\phi_R\mid \mathcal R \} $ and $\{\varphi_R\mid \mathcal R \} $ be collections of functions adapted to $\mathcal R $, with at least one collection having  zeros in the $j $th coordinate for $j=1,2 $.  Then we have the inequality 
\begin{equation}
\label{e.two-withoutbesov} 
\norm \sum _{R\in\mathcal R} \alpha(R)\, \varphi_R\otimes \phi_R . \schatten p .  {}\lesssim{}\norm \alpha(\cdot ).  \ell ^p. . 
\end{equation}

\noindent In the Haar case, we have for $\epsilon,\delta\in\{0,1\}^2 $, with the same assumption on zeros, that the estimate holds
\begin{equation}  
\label{e.two-Haarwithoutbesov} 
\norm \sum _{R\in\mathcal R} \alpha(R)\,  h^\epsilon_R \otimes h^\delta_R. \schatten p .  {}\simeq{} \norm \alpha(\cdot ) . \ell ^p . .
\end{equation}
\end{theorem} 
 
\subsection{The Proof of Theorem \ref{t.two-withoutbesov}: The Haar Case} 
\subsubsection{The Case of Two Parameters}
Recall that as in \eqref{e.Haar-nd-dim}, for a dyadic rectangle $R=R_1\times R_2 $, we set $h_R(x_1,x_2)=h^0_{R_1}(x_1) h^0_{R_2}(x_2)$.  

We of course immediately have the inequality 
\begin{equation*}
\norm \sum _{R\in\mathcal R}  \alpha(R)\,  h_R \otimes h_R. \schatten p .  {}\simeq{} \norm\alpha(\cdot ) . \ell ^p ..
\end{equation*}
  And, keeping in mind the proof in one parameter, we need a certain extension of this inequality.

To set some notation to capture the role of zeros, or their absence, we set 
\begin{equation*}
h_R^{(\epsilon_1,\epsilon_2)}(x_1,x_2)=h_{R_1}^{\epsilon_1 }(x_1) h_{R_2}^ {\epsilon_2 }(x_2) . 
\end{equation*}

We discuss the equivalence \eqref{e.two-Haarwithoutbesov}.  There are three possible forms of the paraproduct, after taking duality and permutation of coordinates into account.  Of these, the first and simplest case is 
\begin{equation*}
\sum _{R\in\mathcal R} \alpha(R)\, h^{(1,0)}_R \otimes h_R .
\end{equation*}
This is very clearly a simple variant of the one parameter version, and we do not discuss it. 
   
The second case is  that  $\epsilon=(1,1) $ and $\delta=(0,0)$, i.e.,
\begin{equation*}
\sum _{R\in\mathcal R} \alpha(R)\, h^{(1,1)}_R \otimes h_R .
\end{equation*}
It is clear that is one function is a Haar function, and the other is a  normalized  indicator function.  This case is the most natural analog of the  one parameter case.

 From \eqref{e.expandindicator} we see that for $j=1,2$,
 \begin{equation} \label{e.expandindicator product}
  h^1_{R_j}=\sum _{S_j:\ R_j\varsubsetneqq S_j } \nu_{R_j,S_j} \sqrt{\tfrac{\displaystyle \abs{R_j}}{\displaystyle \abs {S_j}} }   \, h_{S_j}^0.
  \end{equation}
 
Define, for a multi-integer $m=(m_1,m_2)\in \mathbb N^2 $, the collections 
\begin{equation*}
\mathcal R(m,S)=\{R\in\mathcal R\mid R\subset S,\ 2^{m_j} \abs{R_j}=\abs {S_j},\ j=1,2 \}. 
\end{equation*}
There is a corresponding operator $\operatorname A_m$ defined as follows.
\begin{equation} 
\label{e.Adef} 
\operatorname A_m  {}\eqdef{}\sum _{S\in\mathcal R} h_S\otimes H_{m,S},
\end{equation}
where
\begin{equation} 
\label{e.Hdef} 
H_{m,S}  {}\eqdef{} \sum _{R\in\mathcal R(m,S) } \nu_{R,S} \,\alpha(R) h_R . 
\end{equation}
Here, $\nu_{R,S}=\nu_{R_1,S_1}\cdot\nu_{R_2,S_2}$.  Observe that, just as in \eqref{e.Sm<}, we have the estimate 
\begin{equation} \label{e.Am<}
\norm \operatorname A_m.\schatten p.  {}\lesssim{} 2 ^{\delta(p)(m_1+m_2)} \norm \alpha(\cdot) .\ell^p. .
\end{equation}
Here, $\delta(p)=\max\left(0,\frac12-\frac1p\right)$ as before.  This follows just as before, namely observe that the functions $H_{m,S} $ are orthogonal in $S $, and that 
\begin{align*}
\norm H_{m,S}.2. & {}=\left[ \sum _{R\in\mathcal R(m,S) } \abs {\alpha(R)}^2\right]^{1/2} 
{}\le{} 2^{\delta(p)(m_1+m_2)} \left[ \sum _{R\in\mathcal R(m,S) } \abs {\alpha(R)}^p\right]^{1/p}.
\end{align*}

In addition, by \eqref{e.expandindicator product}, we have
\begin{equation*}
\sum _{R\in\mathcal R} \alpha(R)\, h^{(1,1)}_R \otimes h_R  {}={} 
\sum _{m\in \mathbb N^2 } 2 ^{-\frac12(m_1+m_2) } \operatorname A_m 
\end{equation*}
for appropriate choices of signs in the definition of $\operatorname A_m $.  The remainder of the proof is just as in the one parameter case.  For $1\le{}p<\infty $, the Schatten norm obeys the triangle inequality, hence, together with \eqref{e.Am<}, we can estimate 
 \begin{equation} \label{e.normdecay product}
 \begin{split}
 \norm \sum _{R\in\mathcal R} \alpha(R) h^{(1,1)}_R\otimes h_R . \schatten p.&{}\le{}
 {} \sum _{m_1=1}^\infty\sum _{m_2=1}^\infty 2 ^{-\frac12(m_1+m_2) } \norm \operatorname A_m .\schatten p . 
 \\ & \le{} \norm \alpha(\cdot).\ell^p . \sum _{m_1=1}^\infty 2^{-m_1 \min\left(\frac12,\frac1p\right)} \sum _{m_2=1}^\infty 2^{-m_2 \min\left(\frac12,\frac1p\right)} 
 \\ & {}\lesssim{}\norm \alpha(\cdot).\ell^p . .
 \end{split}
 \end{equation}
 In the case that $0<p<1 $, we can rely upon the subadditivity as in \eqref{e.schattentrianglep}, and a very similar 
 argument finishes the proof of the upper bound as above.


 We turn to the case with no proper analog in the one parameter setting, namely 
 \begin{equation}\label{T alpha product}
T_\alpha:= \sum _{R\in\mathcal R} \alpha(R)\, h^{(1,0) }_R\otimes h^{(0,1)}_R .
 \end{equation}
  
 Of course we want to apply \eqref{e.expandindicator} to the first function in the first coordinate, and the 
 second function in the second coordinate.   Doing so suggests these definitions.  
 For $m\in \mathbb N^2 $, and $S\in\mathcal R $, 
 \begin{align}\label{R m S}
 \mathcal R(m,S)&{} {}\eqdef{}\{ R\in\mathcal R\mid R\subset S,\ 2 ^{m_j}\abs {R_j}=\abs {S_j},\ j=1,2\},      
 \\[9pt]
 \operatorname A_m & {}\eqdef{}\sum _{S\in\mathcal R}  \operatorname A _{m,S},\label{A m}
 \end{align}
where
 \begin{align}\label{A m S}
 \operatorname A_{m,S} & {}\eqdef{}	\sum _{R\in\mathcal R(m,S) }	\nu_{R,S} \, \alpha(R) \,  h^{0 }_{S_1}(x_1)h ^{0 }_{R_2 }(x_2)
   \cdot h^{0 }_{R_1}(y_1)h^{0 } _{S_2 }(y_2).
 \end{align}
 Here, $\nu_{R,S} $ are appropriate choices of signs. 
 
  The principle observation is that Proposition~\ref{p.mn} applies to the operators $\operatorname A _{m,S} $, giving 
  the estimate 
  \begin{equation*}
  \norm \operatorname A_{m,S} .2 .  {}\lesssim{} 2^{\delta(p)(m_1+m_2)} 
    \left[ \sum _{R\in\mathcal R(m,S) } \abs{ \alpha(R) }^p \right]^{1/p} . 
  \end{equation*}
 It is easy to see that we then have 
 \begin{equation*}
 \norm \operatorname A_m .\schatten p. {}\lesssim{} 2^{\delta(p)(m_1+m_2)}  \norm \alpha(\cdot) .p..
 \end{equation*}

Finally,  we have 
\begin{align*}
&\sum _{R\in\mathcal R} \alpha(R)\, h^{(1,0) }_R(x_1,x_2)\otimes h^{(0,1)}_R(y_1,y_2) \\
&=\sum _{\substack{R=R_1\times R_2\\ R_1,R_2\in\mathcal D}} \alpha(R) \bigg[\sum _{S_1:\ R_1\varsubsetneqq S_1} \nu_{R_1,S_1} \sqrt{\tfrac{\displaystyle \abs{R_1}}{\displaystyle \abs {S_1}} }   \, h_{S_1}^0(x_1)\bigg] h^{0 }_{R_2}(x_2)\, h^{0 }_{R_1}(y_1) \\
&\qquad\times\bigg[\sum _{S_2:\ R_2\varsubsetneqq S_2} \nu_{R_2,S_2} \sqrt{\tfrac{\displaystyle \abs{R_2}}{\displaystyle \abs {S_2}} }   \, h_{S_2}^0(y_2)\bigg]  \\
&=\sum _{\substack{S=S_1\times S_2\\ S_1,S_2\in\mathcal D}} \sum _{\substack{R_1\subsetneqq S_1\\ R_2\subsetneqq S_2\\ R_1,R_2\in\mathcal D}}  \alpha(R) \nu_{R,S}\ h^{0 }_{R_1}(y_1) h^{0 }_{R_2}(x_2)\,  \ \bigg[  \sqrt{\tfrac{\displaystyle \abs{R_1}}{\displaystyle \abs {S_1}} }   \, h_{S_1}^0(x_1)\bigg] \bigg[ \sqrt{\tfrac{\displaystyle \abs{R_2}}{\displaystyle \abs {S_2}} }   \, h_{S_2}^0(y_2)\bigg]  \\
&= \sum _{m\in \mathbb N^2 } 2 ^{-\frac12 (m_1+m_2) } \operatorname A_m .
\end{align*}
The conclusion of this case is just as that of \eqref{e.normdecay product}.  Thus, we obtain that
$$ \norm T_\alpha. \schatten p .  {}\lesssim{} \norm \alpha(\cdot ) . \ell ^p ., \qquad 0<p<\infty,
 $$
 where $T_\alpha$ is as defined in \eqref{T alpha product}.

Combining all these cases above, we obtain that for $(\epsilon,\delta) = ((0,0),(0,0))$, 
$((1,1),(0,0))$, $((0,0),(1,1))$, $((1,0),(0,1))$ and $((0,1),(1,0))$
 $$
 \norm \sum _{R\in\mathcal R} \alpha(R)\,  h^\epsilon_R \otimes h^\delta_R. \schatten p .  {}\lesssim{} \norm \alpha(\cdot ) . \ell ^p ., \qquad 0<p<\infty .
 $$
 
We now prove the reverse direction. We only pick one case, $(\epsilon,\delta) = ((1,0),(0,1))$,  to discuss the details, that is, we aim to prove that
 \begin{align}
 \norm \operatorname T_\alpha.\schatten p.  {}\gtrsim{} \norm \alpha(\cdot ) . \ell ^p ., \qquad 0<p<\infty,
 \end{align}
 where $T_\alpha$ is as defined in \eqref{T alpha product}.

To begin with, we fix $\alpha(\cdot)$ such that $\|\alpha(\cdot)\|_{\ell^p}=1$.   We want to show that  $$\norm \operatorname T_\alpha.\schatten p. {}\gtrsim{}1,$$  
 
 A collection of dyadic intervals 
 $\mathcal D_{\ell} $ will \emph{have scales separated by $\ell$ } if 
  it satisfies the conditions 
  $I\in \mathcal D_{\ell}  $ implies $I\pm2\abs {I}\in \mathcal D_{\ell} $, but $I\pm\abs {I} \not\in \mathcal D_{\ell} $ and  $\{\log_2 \abs 
  I\mid I\in\mathcal D_{\ell}\}  {}={} a+\ell \mathbb Z $ for some choice of integer $a $. 
 All dyadic intervals $\mathcal D $ are a union of $2\ell $ subcollections with scales separated by $\ell $.  We denote by $\mathcal D_{\ell}^{(1)} $  such collection of intervals in the first variable, and by $\mathcal D_{\ell}^{(2)} $  such collection of intervals in the second variable.
 Let $\big(\mathcal D_\ell^{(1)}\big)' $ be the collection of \emph{parents}  of  those dyadic intervals in $\mathcal D_\ell^{(1)} $  (the parent of $I $ is the next  dyadic larger interval) and let $\big(\mathcal D_\ell^{(2)}\big)' $ be the collection of \emph{parents}  of  those dyadic intervals in $\mathcal D_\ell^{(2)} $.

 For an integer $\ell $ to be chosen,  we can separate the scales in the 
 dyadic intervals in each variable, choosing $\mathcal D_{\ell}^{(1)} $ and $\mathcal D_{\ell}^{(2)} $  so that 
 \begin{equation}  \label{e.loweralpha} 
  \sum _{\substack{R=R_1\times R_2\in\mathcal R,\\ R_1\in\mathcal D_{\ell}^{(1)}, R_2\in \mathcal D_{\ell}^{(2)} }} \abs{\alpha(R)}^p\ge(4\ell^2)^{-p}. 
  \end{equation} 
 Let $\widetilde{\operatorname P}_\ell $ be the projection onto the span of the functions $\big\{ h_S\mid S=S_1\times S_2\in \mathcal D_\ell^{(1)}\times \big(\mathcal D_\ell^{(2)}\big)'  \big\} $, i.e.,
 $$ \widetilde{\operatorname P}_\ell (f) = \sum_{S\in \mathcal D_\ell^{(1)}\times \big(\mathcal D_\ell^{(2)}\big)' } \langle f,h_S^0\rangle h_S^0. $$ 
 Let $\widetilde{\operatorname P}'_\ell $ be the corresponding projection onto the span of the functions $\big\{ h_S\mid S=S_1\times S_2\in \big(\mathcal D_\ell^{(1)}\big)' \times \mathcal D_\ell^{(2)} \big\} $, i.e.,
 $$ \widetilde{\operatorname P}_\ell (f) = \sum_{S\in \big(\mathcal D_\ell^{(1)}\big)' \times \mathcal D_\ell^{(2)}  } \langle f,h_S^0\rangle h_S^0. $$ 
 Appealing to Proposition~\ref{p.ortholower} we  have the estimate 
 \begin{equation*}
 \norm \operatorname T_\alpha.\schatten p.\ge\norm \widetilde{\operatorname P}_\ell' \operatorname T_\alpha \widetilde{\operatorname P}_\ell
 .\schatten p.  .
 \end{equation*}
 We shall show that  $\norm \widetilde{\operatorname P}_\ell' \operatorname T_\alpha \widetilde{\operatorname P}_\ell
 .\schatten p.$ is at least $(16 \ell^2)^{-1} $  for $\ell $ sufficiently large 
 depending only on $p $. 
 
 Define
 \begin{equation*}
 \operatorname T^{(1,1)}_\alpha {}\eqdef{}  \sum _{\substack{R=R_1\times R_2\in\mathcal R,\\ R_1\in\mathcal D_{\ell}^{(1)}, R_2\in \mathcal D_{\ell}^{(2)} }} \nu_{R,R'}\ \alpha(R)\, h _{R' }(x_1,y_2)\, h_R(y_1,x_2),
 \end{equation*}
 where $R'=R'_1\times R'_2 $ with $R'_i$ the parent of $R_i $ for $i=1,2$. Then the $\mathbb S^p$ norm can be calculated exactly from \eqref{e.schatten-def} as follows
 \begin{equation*}
 \norm \operatorname  T^{(1,1)}_\alpha.\schatten p.^p= \sum _{\substack{R=R_1\times R_2\in\mathcal R,\\ R_1\in\mathcal D_{\ell}^{(1)}, R_2\in \mathcal D_{\ell}^{(2)} }}   |\alpha(R)|^p {}\ge{}(4\ell^2)^{-p},
 \end{equation*}
 where the above inequality follows from \eqref{e.loweralpha}.
 This is the main term in providing a lower bound on the Schatten norm for $\operatorname T_\alpha $.
 
 We appeal to some of the estimates used in the proof of the upper bound above.  
Define 
\begin{equation*}
\operatorname  T^m_\alpha {}\eqdef{} \sum_{S=S_1\times S_2,\ S_1\in  (\mathcal D_\ell^{(1)})',\ S_2\in (\mathcal D_\ell^{(2)})'}   \widetilde{\operatorname  A} _{m\ell+1,S},\qquad m=(m_1,m_2)\in\mathbb N^2,
\end{equation*}
 where
 \begin{align}\label{widetilde A m S}
 \widetilde{\operatorname  A}_{m,S} & {}\eqdef{}	\sum _{R\in \widetilde{\mathcal R}(m,S) }	\nu_{R,S} \, \alpha(R) \,  h^{0 }_{S_1}(x_1)h ^{0 }_{R_2 }(x_2)
   \cdot h^{0 }_{R_1}(y_1)h^{0 } _{S_2 }(y_2)
 \end{align}
and
 \begin{align*}
 \widetilde{\mathcal R}(m,S)&{} {}\eqdef{}\{ R=R_1\times R_2\in\mathcal R\mid R_1\in\mathcal D_{\ell}^{(1)}, R_2\in \mathcal D_{\ell}^{(2)} , R\subset S,\ 2 ^{m_j}\abs {R_j}=\abs {S_j},\ j=1,2\},      
 \end{align*}

 Note that 
 \begin{align*}
 \widetilde{\operatorname P}_\ell' \operatorname T_\alpha \widetilde{\operatorname P}_\ell
 &=\sum_{\substack{S=S_1\times S_2\\ S_1\in  (\mathcal D_\ell^{(1)})'\\ S_2\in (\mathcal D_\ell^{(2)})'}}
\sum_{\substack{ R=R_1\times R_2\\ R_1\in  \mathcal D_\ell^{(1)},\ R_1\subsetneq S_1\\ R_2\in  \mathcal D_\ell^{(2)},\ R_2\subsetneq S_2}} \alpha(R)\ \nu_{R,S}\ \sqrt{ |R|\over |S| }\ h^0_{S_1}(x_1)h^0_{R_2}(x_2)  h^0_{R_1}(y_1)h^0_{S_2}(y_2)\\
&={1\over2}\operatorname T^{(1,1)}_\alpha + {1\over\sqrt2} \sum _{m_2\geq2} 2 ^{-\frac12(m_2\ell+1)}\operatorname T^{(1,m_2)}_\alpha + {1\over\sqrt2} \sum _{m_1\geq2} 2 ^{-\frac12(m_1\ell+1)}\operatorname T^{(m_1,1)}_\alpha \\
&\qquad+ \sum _{\substack{m\in\mathbb N^2\\ m_1,m_2\geq2}} 2 ^{-\frac12(m_1\ell+1+m_2\ell+1)}\operatorname T^m_\alpha \\
&=: Term_1+Term_2+Term_3+Term_4.
 \end{align*}

 Then similar to the estimates in \eqref{e.normdecay product}, we have the estimate 
 \begin{equation*}
 \norm Term_i .\schatten p. {}\lesssim{}2^{-\min\left(\frac12,\frac1p\right) \ell },\qquad i=2,3,4.
 \end{equation*}
 The implied constant depends upon $0<p<\infty $.  Therefore, for an absolute choice of $\ell $, we will have the estimate 
 $\norm \operatorname T_\alpha .\schatten p . \ge(16\ell^2)^{-1} $.  This completes the proof of \eqref{e.loweralpha}.

 \bigskip
\subsubsection{The Case of Higher Parameters} 
In the general $n$ parameter setting, the number of paraproducts increases dramatically. The only restriction in forming a paraproduct is that in each of the $n $ coordinates,  there must be zeros in one  of the three classes of functions corresponding to the paraproduct.

Let us pass immediately to the Haar case, without mention of the Besov norms.  Given $\epsilon=(\epsilon_1,\ldots,\epsilon_n)\in\{0,1\}^n $, set 
\begin{equation}  
\label{e.Haardfunctions}
h_R^\epsilon(x_1,\ldots ,x_n)=\prod _{j=1}^n h_{R_j}^{\epsilon_j} (x_j ) . 
\end{equation}
In this case, the proper statement of Theorem \ref{t.two-withoutbesov} is 

\begin{theorem}
\label{t.Haardparproducts}
Given $\epsilon,\delta\in \{0,1\} ^n $, we assume that in each coordinate $j$,  $\epsilon_j \cdot \delta_j =0$.  Then, we have the equivalence 
\begin{equation}  
\label{e.Haardparproducts}
\norm \sum _{R\in\mathcal R } \alpha(R) \, h_R^\epsilon\otimes h^\delta_R .\schatten p . {}\simeq{}\norm \alpha(\cdot) .\ell ^p. .
\end{equation}
\end{theorem}

The proof of Theorem \ref{t.Haardparproducts} is again similar to what appeared above.  Let $\mathbb O_j $ denote the coordinates in which $\epsilon $ equals $j$.  These sets of coordinates are disjoint, and not necessarily all of $\{1,\ldots, n \} $. 

The case that both $\mathbb O_j $, for $j=0,1 $, are empty, is trivial.  The case that one of these two is the empty set is (essentially) the one parameter case, but this will fit into the discussion below.   Assume that $\mathbb O_1 $ is not empty.   Let $\mathbb O {}\eqdef{}\mathbb O_0\cup \mathbb O_1$.

We define, for choices of $S\in\mathcal R $, and $m\in \mathbb N^{\mathbb O}  $, 
\begin{align*}
\mathcal R(m,S) & {}\eqdef{}\{R\in\mathcal R \mid R_j\subset S_j,\ 2^{m_j}\abs {R_j}=\abs{S_j} ,\ j\in \mathbb O\},  
\\[5pt]
\operatorname A_m & {}\eqdef{} \sum _{S\in\mathcal R}  \operatorname A _{m,S}, 
\\[5pt]
\operatorname A  _{m,S }&  {}\eqdef{}\sum _{R\in\mathcal R(m,S) } \nu_{R,S} \alpha(R)\, 
H _{0,m,R,S}\otimes H _{1,m,R,S}, 
\\[5pt]
H _{j,m,R,S} &{}\eqdef{} \prod _{j\in \mathbb O_j} h _{S_j}(x_j)\prod _{k\not\in \mathbb O_j} h _{R_k}(x_k)  
,\qquad  j=0,1.
\end{align*}

The main point is that Proposition~\ref{p.mn} applies to the operators $\operatorname A _{m,S} $, giving the estimate 
\begin{equation*}
\norm \operatorname A_{m,S} .2.  {}\lesssim{}2 ^{\delta(p) \sum _{\mathbb O} m_j }   
\left[\sum _{R\in\mathcal R(m,S) } \abs{\alpha(R)}^p\right]^{1/p}. 
\end{equation*}
And this, plus a simple argument, gives us the estimate 
\begin{equation*}
\norm \operatorname A_{m} .\schatten p.  {}\lesssim{}2 ^{\delta(p) \sum _{\mathbb O} m_j }  \norm \alpha(\cdot).\ell ^p . .
\end{equation*} 
Finally, it is the case that for appropriate choices of signs $\nu_{R,S} $, we have 
\begin{equation*} 
\sum _{R\in\mathcal R } \alpha(R) \, h_R^\epsilon\otimes h^\delta_R  
{}={}\sum _{m\in \mathbb N ^{\mathbb O } } 2 ^{-\frac12 \sum_{\mathbb O} m_j  }\operatorname A _{m} .
\end{equation*}
This is clearly enough to prove the theorem for both the upper bound and lower bound of \eqref{e.Haardparproducts} following the arguments in Section 5.2.1. 

\subsection{The Proof of Theorem \ref{t.two-withoutbesov}: The Wavelet Case}

We can derive the same argument using similar steps as in the Haar setting in Section 5.2, combining with the wavelet basis and the almost orthogonality estimate as used in the one-parameter setting in Section   4.2.

\section{Concluding Remarks}
\label{s.concrem}

We close this paper by making several concluding remarks.  First, we chose to work with the Haar/wavelet bases for $L^2(\mathbb R)$, and then the appropriate generalization of these to the tensor product setting for $L^2(\mathbb R^n)$.  It is possible to formulate and prove analogous results by starting with the Haar/wavelet basis for $L^2(\mathbb R^{n_j})$ and then forming the $d$-parameter tensor product bases for $L^2(\mathbb R^{\sum_{j=1}^d n_j})$.  The only additional difficulty that this produces is more complicated notation.  The interested reader can easily modify the results appearing in this paper to arrive at these more general theorems.

In this paper we have focused on continuous and dyadic paraproducts.  As mentioned in the introduction, paraproducts are an equivalent way to view Hankel operators.  Once estimates for the paraproducts are known, then using a nice result of S. Petermichl \cite{stefanie} it is possible to recover the Hankel operator from an an averaging of a dyadic shift operator.  Using these results one can prove that a  Hankel operator belongs to the Schatten class $\schatten p$ if the symbol belongs to the Besov space $\besov.$.  It is carried out very nicely in the paper \cite{MR2097606} of S. Pott and M. Smith in their discussion of $\schatten p$ class, $1<p<\infty$, of one-parameter dyadic paraproducts, and the related Besov spaces.

Equivalent to the Hankel operator,
a similar statement holds for the commutator between multiplication by a symbol $b$ and the Hilbert transform (and the appropriate multiparameter analogues).  The very recent work by the authors \cite{LLW} exploited the paraproducts to establish this commutator statement in the one-parameter two-weight setting.

\bigskip
\bigskip
\bigskip
\noindent {\bf Acknowledgement}: Lacey is supported in part by National Science Foundation DMS Grant. B. D. Wick's research is supported in part by National Science Foundation Grants DMS \# 1800057, \# 2054863, and \# 2000510 and Australian Research Council DP 220100285.
Li and Wick are supported by ARC DP 220100285.


\end{document}